\documentclass{article}\usepackage[utf8]{inputenc}

\usepackage{comment,xcomment,amssymb,amsmath,color}
\usepackage{rotating,xypic,array}
\usepackage{latexsym,mathtools,amsthm}
\usepackage{subfig,rotating}
\usepackage{booktabs} % pacchetto per tabelle.
\usepackage{enumitem}
\usepackage{url} %accetta tilde e indirizzi web
\usepackage{longtable}
\usepackage[section]{placeins}
\usepackage[longtable]{multirow}
\usepackage[colorlinks,citecolor=blue]{hyperref}
\usepackage{geometry}

\makeatletter
\def\namedlabel#1#2{\begingroup
\def\@currentlabel{#2}%
\label{#1}\endgroup
}
\makeatother

\title{Pseudo-Riemannian Sasaki solvmanifolds}
\author{Diego Conti and Federico A. Rossi and Romeo Segnan Dalmasso}

\newtheorem{theorem}{Theorem}[section]
\newtheorem{lemma}[theorem]{Lemma}

\newtheorem{corollary}[theorem]{Corollary}
\newtheorem{proposition}[theorem]{Proposition}
\theoremstyle{definition}
\newtheorem{definition}[theorem]{Definition}
\newtheorem{example}[theorem]{Example}

\theoremstyle{remark}
\newtheorem{remark}[theorem]{Remark}

\newcommand{\R}{\mathbb{R}}
\newcommand{\lie}[1]{\mathfrak{#1}} %Lie algebras
\newcommand{\g}{\lie{g}}
\newcommand{\Lie}{\mathcal{L}} %Lie derivative

\newcommand{\C}{\mathbb{C}}

\newcommand{\hook}{\lrcorner\,}
\newcommand{\SU}{\mathrm{SU}}

\newcommand{\su}{\mathfrak{su}}

\newcommand{\SL}{\mathrm{SL}}

\newcommand{\id}{\operatorname{Id}} % the identity

\newcommand{\gl}{\lie{gl}}
\newcommand{\Sl}{\lie{sl}}
\newcommand{\Span}[1]{\operatorname{Span}\left\{#1\right\}}

\newcommand{\eps}{\varepsilon}
\DeclareMathOperator{\ric}{ric} %tensore, ric = lambda g
\DeclareMathOperator{\Der}{Der}
\DeclareMathOperator{\ad}{ad}

\DeclareMathOperator{\Ad}{Ad}
\DeclareMathOperator{\Tr}{tr}

\newcolumntype{C}{>{$}c<{$}}
\newcolumntype{L}{>{$}l<{$}}
\newcolumntype{R}{>{$}r<{$}}

\begin{document}
\maketitle

\begin{abstract}
	We study a class of left-invariant pseudo-Riemannian Sasaki metrics on solvable Lie groups, which can be characterized by the property that the zero level set of the moment map relative to the action of some one-parameter subgroup $\{\exp tX\}$ is a normal nilpotent subgroup commuting with $\{\exp tX\}$, and $X$ is not lightlike. We characterize this geometry in terms of the Sasaki reduction and its pseudo-K\"ahler quotient under the action generated by the Reeb vector field.
	
	We classify pseudo-Riemannian Sasaki solvmanifolds of this type in dimension $5$ and those of dimension $7$ whose K\"ahler reduction in the above sense is abelian.
\end{abstract}

\renewcommand{\thefootnote}{\fnsymbol{footnote}}
\footnotetext{\emph{MSC class 2020}: \emph{Primary} 53C25; \emph{Secondary} 53D20, 53C50, 22E25}
\renewcommand{\thefootnote}{\arabic{footnote}}

\section*{Introduction}
Sasaki manifolds were introduced in~\cite{SasakiHatakeyama} as an odd-dimensional counterpart to K\"ahler geometry; they are characterized by an almost contact metric structure $(\phi,\xi,\eta,g)$ which is both normal and contact. Beside the analogy, they bear a strong relation to K\"ahler geometry in that both the cone over a Sasaki manifold and the space of leaves of the Reeb foliation carry a K\"ahler structure. For pseudo-Riemannian metrics, a completely analogous definition of Sasaki structure can be given, which was first considered in~\cite{Tak:PseudoSasaki}; the relation to pseudo-K\"ahler geometry is the same as in the definite setting.

Arguably, the most interesting Sasaki metrics are those satisfying the Einstein condition $\ric=2n g$, where the Einstein constant is fixed by the dimension. Both in the Riemannian and indefinite case, Einstein-Sasaki metrics are characterized by the existence of a Killing spinor (see~\cite{Baum:TwistorAndKilling}), which makes them relevant for general relativity and supersymmetry (see \cite{Walker1970OnSpacetimes,Duff1986Kaluza-KleinSupergravity}).

In this paper we focus on the homogeneous case, and particularly on invariant pseudo-Riemannian Sasaki metrics on solvmanifolds. Although we do not insist on the Einstein condition here, the prospect of applying the machinery to produce Einstein-Sasaki metrics leads us to consider \emph{standard} solvmanifolds, corresponding to semidirect products $\g\rtimes\lie a$, where $\g$ is nilpotent, $\lie a$ abelian and their sum orthogonal. Indeed, all Riemannian Einstein solvmanifolds are of this type (see \cite{Heber:noncompact,Lauret:Einstein_solvmanifolds}), and even in the indefinite case the standard condition has proved quite effective to produce examples (see \cite{ContiRossi:IndefiniteNilsolitons,ContiRossi:NiceNilsolitons}). In fact, the most studied standard Lie algebras are those of Iwasawa type (or pseudo-Iwasawa, for indefinite signature), namely those for which $\ad X$ is symmetric for all $X$ in $\lie a$.

Restricting to left-invariant pseudo-Riemannian Sasaki metrics on solvable Lie groups allows us to work at the Lie algebra level; we shall therefore refer to the structures under consideration as Sasaki structures on a Lie algebra. Our first result (Proposition~\ref{prop:sasakinotpseudoiwasawa}) is that Sasaki Lie algebras cannot be of pseudo-Iwasawa type. This motivates us to study the more general class of standard Lie algebras, though restricting for simplicity to one-dimensional abelian factors, i.e. $\tilde\g=\g\rtimes\Span{e_0}$. In Proposition~\ref{prop:sasakistandard}, we characterize the Sasaki condition on $\tilde\g$ in terms of the induced structure on $\g$. The resulting conditions on $\g$ are somewhat unwieldy.

However, the situation simplifies if we impose that $\g$ is the zero-level set of a moment map relative to the action of a one-parameter subgroup. In practice, this means that $\phi (e_0)$ lies in the center $\lie z(\g)$. We dub this particular class of Sasaki structures \emph{$\lie z$-standard}. One can then take the Sasaki reduction in the sense of contact geometry, obtaining a new Sasaki nilmanifold endowed with a derivation $D=\ad e_0$ satisfying certain conditions (Corollary~\ref{cor:sasakistandardcentral}). In this setting, the Reeb field $\xi$ is central, so one can take a further quotient and obtain a pseudo-K\"ahler nilmanifold in three dimensions less $(\check \g, \check J,\check \omega)$ (Corollary~\ref{cor:reductiontokahler}); since this quotient can be interpreted as a symplectic reduction of the pseudo-K\"ahler Lie algebra $\tilde\g/\Span{\xi}$, we call it the \emph{K\"ahler reduction} of $\tilde\g$. The K\"ahler reduction admits a derivation $\check D$ induced by $D$, commuting with $\check J$ and satisfying a quadratic equation of the form
\begin{equation}
	\label{eqn:quadraticD}
	[\check D^s,\check D^a]=h\check D^s-2(\check D^s)^2,
\end{equation}
with $h$ a real constant, and $\check D^s,\check D^a$ denoting the symmetric and antisymmetric part of $\check D$ .

This construction can be inverted: starting from a pseudo-K\"ahler nilmanifold with a derivation as above, one obtains a pseudo-K\"ahler solvmanifold in two dimensions higher, then giving a $\lie z$-standard Sasaki solvmanifold by taking a circle bundle (Proposition~\ref{prop:constructive}). This procedure differs from the double extension procedure considered in~\cite{BoucettaTibssirte}, in that the two ``extra'' dimensions span a definite two-plane, rather than neutral.

We show that up to isometry, when $\check D^s$ is both a derivation and diagonalizable over $\C$ it can be assumed to be a projection, giving a simple explicit form to the resulting Sasaki structure (Corollary~\ref{cor:diagonalizable}). Making use of this fact, we classify $\lie z$-standard Sasaki solvmanifolds in dimension $5$ (Theorem~\ref{thm:central5}), and all those in dimension $7$ whose K\"ahler reduction is abelian (Theorem~\ref{thm:central7}).

\smallskip
\textbf{Acknowledgments}
This paper was written as part of the PhD thesis of the third author, written under the supervision of the first author, for the joint PhD programme in Mathematics Università di Milano Bicocca -- University of Surrey.

The authors acknowledge Gruppo Nazionale per le Strutture Algebriche, Geometriche e le loro Applicazioni (GNSAGA) of Istituto Nazionale di Alta Matematica (INdAM).

\section{Pseudo-Riemannian Sasaki structures}
In this section we recall some basic definitions and facts on pseudo-Riemannian Sasaki structures. For further details we refer to~\cite{CalvCaLo:PseudoRiemHomBook,Tak:PseudoSasaki}.
\begin{definition}
	An \emph{almost contact structure} on a $(2n + 1)$-dimensional manifold $M$ is a triple $(\phi, \xi, \eta)$, where $\phi$ is a tensor field of type $(1,1)$, $\xi$ is a vector field, and $\eta$ is a $1$-form, such that
	\[\eta(\xi)=1,\qquad \eta\circ \phi=0,\qquad \phi^2 =-\id+\eta\otimes\xi.\]
	Given a pseudo-Riemannian metric $g$ on $M$, the quadruple $(\phi ,\xi,\eta,g)$ is called \emph{an almost contact metric structure} if $(\phi, \xi, \eta)$ is an almost contact structure and
	\[g(\xi,\xi)=\eps\in\{\pm 1\},\qquad \eta=\eps \xi^{\flat},\qquad g(\phi X,\phi Y)=g(X,Y)-\eps \eta(X)\eta(Y),\]
	for any vector fields $X,Y$.
	
	We will assume $\epsilon=1$ in the sequel.
\end{definition}
Note that if $(\phi,\xi,\eta,g)$ is an almost contact metric structure with $g(\xi,\xi)=\eps=-1$, then defining $\bar g=-g$ we have that $(\phi,\xi,\eta,\bar g)$ is another almost contact metric structure such that $\bar g(\xi,\xi)=\bar \eps=1$, so our assumption does not entail a loss of generality.

\begin{remark}
	The generalized eigenspace of $0$ for $\phi$ is generated by $\xi$. Therefore $0$ is an eigenvalue and $\xi$ is an eigenvector, i.e. $\phi(\xi)=0$.
\end{remark}
\begin{remark}
	The endomorphism $\phi$ is always skew-symmetric: indeed,
	\begin{multline*}
		g(\phi (X),Y)=-g(\phi X, \phi^2Y-\eta(Y)\xi)\\
		=-g(X,\phi (Y))+\eta(X)\eta(\phi(Y))=-g(X,\phi(Y)).
	\end{multline*}
	In fact, if $\phi$ is assumed to be skew-symmetric, $g(\phi X,\phi Y)=g(X,Y)-\eps \eta(X)\eta(Y)$ is equivalent to $\phi^2 =-\id+\eta\otimes\xi$.
\end{remark}

We define the \emph{fundamental $2$-form} associated to the almost contact metric structure $(\phi,\xi, \eta, g)$ as
\[\Phi = g(\cdot, \phi \cdot).\]
In addition, in analogy with the Nijenhuis tensor field for complex manifolds, we define
\[N_{\phi} = \phi^2 [ X , Y ] + [\phi X , \phi Y ] -\phi[\phi X , Y ]- \phi[ X , \phi Y ].\]

\begin{definition}
	An almost contact metric structure $(\phi, \xi, \eta, g)$ is said to be \emph{Sasaki} if
	$(\phi,\xi,\eta,g)$ satisfies $N_{\phi} + d\eta \otimes \xi = 0$ and $d\eta = 2\Phi$.
\end{definition}

Sasaki structures can be characterized in terms of the covariant derivative $\nabla\phi$; as usual, we indicate by $\nabla$ the Levi-Civita connection, by $R$ its curvature tensor, by $\ric$ its Ricci tensor.
\begin{lemma}[{\cite[Proposition 1]{Tak:PseudoSasaki}}]
	\label{lemma:nablaxi}
	Given an almost contact metric structure
	$(\phi, \xi, \eta, g)$ on a manifold of dimension $2n+1$ such that
	\[(\nabla_X\phi)Y =g(X,Y)\xi- \eta(Y)X,\]
	the following hold:
	\begin{enumerate}
		\item $\nabla_X \xi=-\phi(X)$;
		\item $\xi$ is a Killing vector field;
		\item $d\eta(X,Y) = 2\Phi(X,Y)$;
		\item $R(X,Y)\xi=\eta(Y)X-\eta(X)Y$;
		\item $\ric(\xi, X) = 2n \eta (X)$.
	\end{enumerate}
\end{lemma}

Arguing as in \cite[Theorem~7.3.16]{BoGa:SasakianBook}, one obtains:
\begin{proposition}\label{pr:SasakiEquivalent}
	Let $(\phi,\xi,\eta,g)$ be an almost contact pseudo-Riemannian metric structure on $M$. The following are equivalent:
	\begin{enumerate}
		\item $(\phi,\xi,\eta,g)$ is Sasaki;
		\item the cone $(\R^+\times M,J,\omega)$ is pseudo-K\"ahler;
		\item $(\nabla_X\phi)Y =g(X,Y)\xi- \eta(Y)X$;
		\item $\nabla_X\Phi=\eta\wedge X^\flat$.
	\end{enumerate}
\end{proposition}

Pseudo-Sasaki manifolds are related to pseudo-K\"ahler geometry in the following way. Recall that a pseudo-K\"ahler structure on a manifold $M$ is an almost-pseudo-Hermitian structure $(J,g,\omega)$, with the convention that $\omega=g(\cdot,J\cdot)$, such that $J$ is integrable and $\omega$ is closed; equivalently, $\omega$ is parallel with respect to the Levi-Civita connection.

Like in the Riemannian case, we have the following:
\begin{proposition}[\cite{Ogiue:OnFiberings}]
	\label{prop:kahlerquotient}
	Let $M$ have a pseudo-Riemannian Sasaki structure $(\phi,\xi,\eta,g)$. Then the space of leaves of the Reeb foliation has an induced pseudo-K\"ahler structure.
\end{proposition}

Finally, we recall that given a Sasaki structure $(\phi,\xi,\eta,g)$ and a positive constant $a$, we can define another Sasaki structure by
\[\hat\phi=\phi, \qquad \hat\xi=a^{-1}\xi, \qquad \hat\eta=a\eta, \qquad \hat g=ag+(a^2-a)\eta\otimes\eta.\]
Such a transformation is called a \emph{$\mathcal{D}$-homothety}. This defines an equivalence relation between Sasaki structures on a given manifold.

\section{Sasaki Lie algebras}
Throughout the paper, we consider left-invariant structures on Lie groups, which can be characterized at the Lie algebra level. Accordingly, we shall refer to pseudo-Riemannian metrics on a Lie algebra, Sasaki structures etc. to mean objects defined at the Lie algebra level and silently extended to the Lie group by left translation.

Recall from~\cite{ContiRossi:IndefiniteNilsolitons} that a \emph{standard decomposition} on a Lie algebra $\tilde\g$ endowed with a pseudo-Riemannian metric is an orthogonal decomposition $\tilde\g=\g\rtimes\lie a$, with $\g$ nilpotent and $\lie a$ abelian. A standard decomposition is \emph{pseudo-Iwasawa} if $\ad X$ is symmetric for all $X\in\lie a$. These definitions mimick and generalize analogous definitions for Riemannian metrics (see~\cite{Heber:noncompact}), and they have proved useful in the study of Einstein metrics (\cite{ContiRossi:IndefiniteNilsolitons}).

It is well known that nonisomorphic Lie algebras can be isometric, meaning that the corresponding pseudo-Riemannian manifolds are isometric. The method to obtain such isometries is recalled below
in Proposition~\ref{prop:pseudoAzencottWilson}. A natural question is whether one can choose a representative in an isometry class of Sasaki Lie algebras which admits a pseudo-Iwasawa decomposition. We show that this is never the case: indeed, no Sasaki Lie algebras admits a pseudo-Iwasawa decomposition. This will motivate the study of the more general standard case in the following sections.

We begin this section with an example of a standard Sasaki Lie algebra.
\begin{example}\label{example:4.3}
	\label{ex:esempiodr}
	Consider the $5$-dimensional Lie algebra
	\[\g=(0,-2 e^{12}-2 e^{34},-3 e^{45}-e^{13}+3 e^{24},3 e^{35}-3 e^{23}-e^{14},2 e^{12}+2 e^{34});\]
	with notation as in \cite{Salamon:ComplexStructures}; explicitly, we have a fixed basis $\{e_i\}$ of $\g$ such that the dual basis $\{e^i\}$ of $\g^*$ satisfies $de^1=0$, $de^2=-2 e^1\wedge e^2-2 e^{3}\wedge e^4$ and so on, with $d\colon\g^*\to\Lambda^2\g^*$ denoting the Chevalley-Eilenberg operator.
	As observed in {\cite[Example~5.6]{ConDal:KillingSpinHyper}}, the Lie algebra $\g$ carries an Einstein-Sasaki structure given by
	\begin{gather*}
		g=-e^1\otimes e^1-e^2\otimes e^2-e^3\otimes e^3-e^4\otimes e^4+e^5\otimes e^5,\\
		\xi=e_5, \qquad \Phi=e^{12}+ e^{34}.
	\end{gather*}
	This has a standard decomposition $\Span{e_1}\ltimes\Span{e_2,e_3,e_4,e_5}$. Notice that this metric can be obtained from the Riemannian $\eta$-Einstein-Sasaki metric on the Lie algebra $\g_0$ of \cite{AndFinVez:Sasaki5} by reversing the sign of the metric along the Reeb vector field.
\end{example}

Given a Lie algebra $\g$ with a metric $g$, for any endomorphism $f\colon\g\to\g$ we write $f=f^s+f^a$, where $f^s$ is symmetric and $f^a$ is skew-symmetric relative to the metric, i.e.
\[f^s=\frac12(f+f^*), \qquad f^a=\frac12(f-f^*).\]

Consider a semidirect product $\tilde \g=\g\rtimes\lie a$, with $\lie a$ abelian, and fix any metric. In
\cite[Section~1.8]{EberleinHeber} and \cite[Proposition~1.19]{ContiRossi:IndefiniteNilsolitons} it was shown that under certain conditions one can obtain an isometric Lie algebra by projecting on the symmetric part. These results assume that the decomposition is standard; however, the proof holds more generally, without assuming that the metric is standard and taking more general projections:
\begin{proposition}
	\label{prop:pseudoAzencottWilson}
	Let $\tilde \g$ be a pseudo-Riemannian Lie algebra (not necessarily standard) of the form $\tilde{\g}=\g\rtimes \lie{a}$; let $\chi\colon\lie a\to\Der(\g)$ be a Lie algebra homomorphism such that, extending $\chi(X)$ to $\tilde\g$ by declaring it to be zero on $\lie a$,
	\begin{equation}
		\label{eqn:adXstar}
		\chi(X)^s=(\ad X)^s, \qquad [\chi (X),\ad Y]=0,\ X,Y\in\lie a.
	\end{equation}
	Let $\tilde{\g}^*$ be the Lie algebra $\g\rtimes_\chi\lie{a}$. Then there is an isometry between the connected, simply connected Lie groups with Lie algebras $\tilde{\g}$ and $\tilde{\g}^*$, with the corresponding left-invariant metrics, whose differential at $e$ is the identity of $\g\oplus\lie{a}$ as a vector space.
\end{proposition}
\begin{proof}
	Observe that for every $X$ in $\lie a$, $\chi(X)$ is a derivation of $\g$ that commutes with $\ad \lie a$ by~\eqref{eqn:adXstar}, and therefore a derivation of $\tilde \g$. For $X$ in $\lie a$, write $\ad X=A(X)+\chi(X)$, where $A(X)$ is an antisymmetric derivation of $\tilde\g$.
	By construction, $A(X)$ is zero on $\lie a$.
	
	The rest of the proof is identical to \cite[Proposition~1.19]{ContiRossi:IndefiniteNilsolitons}, except that one replaces $(\ad X)^a$ with $A(X)$, and one cannot assume that $\exp \g\exp\lie a$ equals the whole connected, simply-connected group $\tilde G$ with Lie algebra $\tilde\g$; however, it is clear that $\exp A(X)$ fixes the connected subgroup with Lie algebra $\lie a$, which is what is needed.
\end{proof}

As a consequence we have a result analogous to \cite[Proposition~1.19]{ContiRossi:IndefiniteNilsolitons} for nonstandard metrics:
\begin{corollary}
	\label{Cor:EasyPseudoAzencottWilson}
	Let $\tilde \g$ be a pseudo-Riemannian Lie algebra of the form $\tilde{\g}=\g\rtimes \lie{a}$; suppose that, for every $X$ in $\lie{a}$, $(\ad X)^*$ is a derivation of $\tilde\g$ vanishing on $\lie a$, and furthermore
	\begin{equation}\label{eqn:adXstarSymmetric}
		[(\ad X)^*,\ad Y]=0, \quad X,Y\in\lie a.
	\end{equation}
	Define $\chi\colon\lie a\to\Der(\g)$ as $\chi(X)=(\ad X)^s$. Let $\tilde{\g}^*$ be the solvable Lie algebra $\g\rtimes_{{\chi}}\lie{a}$.
	
	Then there is an isometry between the connected, simply connected Lie groups with Lie algebras $\tilde{\g}$ and $\tilde{\g}^*$, with the corresponding left-invariant metrics, whose differential at $e$ is the identity of $\g\oplus\lie{a}$ as a vector space.
\end{corollary}

\begin{example}
	\label{ex:43nontrivialcenter}
	We can apply Proposition~\ref{prop:pseudoAzencottWilson} to Example~\ref{example:4.3} with $\lie a=\Span{e_5}$,
	$\lie g=\Span{e_1, \linebreak[0] e_2-e_5,\allowdisplaybreaks e_3,e_4}$ to obtain an isometric Lie algebra
	\begin{gather*}
		\tilde\g=(0,-2 e^{12}-2 e^{34},- e^{13},- e^{14},2 e^{12}+2 e^{34}),\\
		g=-e^1\otimes e^1-e^2\otimes e^2-e^3\otimes e^3-e^4\otimes e^4+e^5\otimes e^5,\\
		\xi=e_5, \qquad \Phi=e^{12}+ e^{34}.
	\end{gather*}
	This can be written as $\Span{e_2,e_3,e_4,e_5}\rtimes\Span{e_1}$, with
	\[\Span{e_2,e_3,e_4,e_5}\cong (-2E^{23},0,0,2E^{23})\]
	and
	\[\ad e_1=2e^2\otimes (e_2-e_5) +e^3\otimes e_3+e^4\otimes e_4.\]
	This is standard but not pseudo-Iwasawa, consistently with Proposition~\ref{prop:sasakinotpseudoiwasawa} below.
\end{example}

In the following, we will need the explicit formula for the Levi-Civita connection of a metric on a Lie algebra, namely
\begin{equation}
	\label{eqn:lc}
	\nabla_w v=-\ad(v)^s w-\frac12 (\ad w)^* v.
\end{equation}
The formula follows immediately from the Koszul formula. In order to specialize to the standard case, we will need to fix an orthogonal basis $\{e_s\}$ on the abelian factor $\lie a$ such that $\tilde g(e_s,e_s)=\epsilon_s$.
\begin{lemma}
	\label{lemma:lcpseudoiwasawa}
	Let $\tilde\g$ be a Lie algebra with a standard decomposition $\tilde\g=\g\oplus\lie{a}$. Then
	\[\widetilde\nabla_HX=\widetilde{\ad}(H)^a(X), \qquad \widetilde\nabla_X H=-\widetilde{\ad}(H)^s(X),
	\]
	for all $H\in\lie a$, $X\in\tilde\g$. In addition, if $\{e_i\}$ is an orthogonal basis of $\lie a$ and $v,w\in\g$, we have
	\[\widetilde\nabla_w v = -\ad(v)^s w-\frac12 (\ad w)^* v +\sum_s \epsilon_s \tilde g(\widetilde\ad(e_s)^s v,w)e_s, \quad v,w\in\g.\]
\end{lemma}
\begin{proof}
	If we apply~\eqref{eqn:lc} to $\widetilde\nabla$, we get
	\begin{gather*}
		\begin{split}
			\widetilde\nabla_H X&=-\widetilde \ad(X)^sH-\frac12(\widetilde \ad H)^*X\\
			&=-\frac12 \widetilde \ad(X)H-\frac12 \widetilde \ad(X)^*H-\frac12 \widetilde \ad(H)^*X=\widetilde{\ad}(H)^a(X),
		\end{split}\\
		\widetilde\nabla_X H=-\widetilde \ad(H)^sX-\frac12(\widetilde \ad X)^*H=-\widetilde \ad(H)^sX.
	\end{gather*}
	Now observe that $\widetilde\ad( v)^*w=\ad(v)^*w + \sum_s\epsilon_s\tilde g([v,e_s],w)e_s$. Therefore,
	\begin{align*}
		\widetilde\nabla_w v&=-\frac12\widetilde\ad(v)w-\frac12\widetilde\ad(v)^*w-\frac12\widetilde\ad( w)^*v\\
		&=-\frac12\ad(v)w-\frac12\ad(v)^*w-\frac12\ad( w)^*v\\
		&\quad-\frac12\sum_s\epsilon_s\tilde g([v,e_s],w)e_s
		-\frac12\sum_s\epsilon_s\tilde g([w,e_s],v)e_s\\
		&=-\ad(v)^sw-\frac12\ad( w)^*v\\
		&\quad+\frac12\sum_s\epsilon_s\bigl(\tilde g(\ad (e_s)v,w)+\tilde g(\ad(e_s)^*v,w)\bigr)e_s.\qedhere
	\end{align*}
\end{proof}

We can now prove the following:
\begin{proposition}
	\label{prop:sasakinotpseudoiwasawa}
	Let $\tilde\g$ be a solvable Lie algebra with a Sasaki pseudo-Riemannian metric $g$. Then there is no pseudo-Iwasawa decomposition.
\end{proposition}
\begin{proof}
	Assume for a contradiction that $\tilde\g=\g\oplus\lie a$ is a pseudo-Iwasawa decomposition. Then by Lemma~\ref{lemma:lcpseudoiwasawa} and Lemma~\ref{lemma:nablaxi} we have
	\[0=\widetilde\nabla_H\xi=-\phi(H), \quad H\in\lie a.\]
	This implies that $\lie a$ is one-dimensional and spanned by $\xi$. We have
	\[-\phi X=\widetilde\nabla_X \xi =-\widetilde\ad(\xi)X.\]
	However $\phi$ is skew-symmetric, while $\widetilde{\ad}(\xi)$ is symmetric, giving a contradiction.
\end{proof}

\section{Sasaki structures on rank-one standard Lie algebras}
In this section we consider standard decompositions of rank one, meaning that the abelian factor $\lie a$ is one-dimensional. Accordingly, $\tilde \g$ will be a solvable Lie algebra endowed with a standard decomposition $\g\rtimes_D\Span{e_0}$, with $D$ a derivation of $\g$ and $\ad e_0=D$; we will denote by $[\,,\,]$ and $d$ the Lie bracket and exterior derivative on $\g$.

\begin{lemma}
	\label{lemma:acstandard}
	Let $\g$ be a nilpotent Lie algebra with a pseudo-Riemannian metric $g$, let $D$ be a derivation, and let $\tau=\pm1$. Then $\tilde\g=\g\rtimes_D\Span{e_0}$ has an almost contact metric structure $(\phi,\xi,\eta,\tilde g)$ such that
	\[\tilde g=g+\tau e^0\otimes e^0, \qquad\widetilde\nabla \xi=-\phi\]
	if and only if $\xi\in\g$ and, writing $b=D^a(\xi)$, for all $u,w\in\g$
	\begin{gather}\label{eqn:acphi}
		\phi(w)=\frac12 (\ad w)^*(\xi)+\tau g(b,w)e_0, \qquad \phi(e_0)=-b,\\
		\label{eqn:acxi} D(\xi)=0, \qquad (\ad \xi)^s=0,\qquad (\ad b)^*(\xi)=0,\\
		\label{eqn:acg} g(w,u)=g(\xi,w)g(\xi,u)+\tau g(b,w)g(b,u)+\frac14g( (\ad w)^*\xi,(\ad u)^*\xi).
	\end{gather}
\end{lemma}
\begin{proof}
	Given $\tilde g=g+\tau e^0\otimes e^0$ and $\xi\in\tilde\g$, define $\eta=\xi^\flat$ and $\phi=-\tilde\nabla\xi$.
	
	Write
	\[\xi=v+ae_0, \quad v\in\g, a\in\R.\]
	By Lemma~\ref{lemma:lcpseudoiwasawa}, we have
	\begin{gather*}
		\widetilde\nabla_w\xi = \widetilde\nabla_wv + a\widetilde\nabla_w e_0=
		-\ad(v)^sw-\frac12(\ad w)^*v+\tau\tilde g(D^s(w),v)e_0 - aD^s(w),\\
		\widetilde\nabla_{e_0}\xi=D^a(v).
	\end{gather*}
	Since $\widetilde\phi(X)=-\widetilde\nabla_X\xi$, we can write
	\begin{gather*}
		\phi(w)=\ad(v)^sw+\frac12(\ad w)^*v -\tau\tilde g(D^s(w),v)e_0 + aD^s(w),\\
		\phi(e_0)=-D^a(v).
	\end{gather*}
	This determines an almost-contact metric structure if
	and only if $\phi$ is skew-symmetric and
	\begin{equation}
		\label{eqn:gtildealmostcontact}
		\tilde g(X,Y)-\eta(X)\eta(Y)=\tilde g(\phi X,\phi Y).
	\end{equation}
	The skew-symmetric condition implies
	\[
	0=\tilde g(\phi(w),e_0)+\tilde g(\phi(e_0),w)
	=-\tau^2\tilde g(D^s(w),v)-\tilde g(D^a(v),w)=-\tilde g(D(v),w)\]
	for all $w$ in $\g$, giving $D(v)=0$. In addition,
	\begin{multline*}
		0=\tilde g(\phi(w),u)+\tilde g(\phi(u),w)\\
		=g(\ad(v)^sw,u)+g(\ad(v)^su,w)+\frac12g((\ad w)^*v,u)\\
		+\frac12g((\ad u)^*v,w)
		+ ag(D^s(w),u) + ag(D^s(u),w)\\
		=2g(\ad(v)^sw,u)+2ag(D^s(w),u),
	\end{multline*}
	giving $\ad(v)^s+aD^s=0$ and
	\[\phi(w)=\frac12 (\ad w)^*(v)-\tau g(D^s(v),w)e_0
	=\frac12 (\ad w)^*(v)+\tau g(D^a(v),w)e_0.
	\]
	Evaluating~\eqref{eqn:gtildealmostcontact} on $w,e_0$ we get
	\begin{multline*}
		-a\tau g(v,w)=\tilde g(w,e_0)-\eta(w)\eta(e_0)=\tilde g(\phi (w),\phi (e_0))\\
		=\tilde g(\frac12 (\ad w)^*(v)+\tau g(D^a(v),w)e_0,-D^a(v))\\
		= g(\frac12 (\ad w)^*v+\tau g(D^a(v),w)e_0,-D^a(v))\\
		=-\frac12 g( (\ad w)^*v,D^a(v))=-\frac12 g(v,[w,D^a(v)]=\frac12g(w,(\ad D^a(v))^*v).
	\end{multline*}
	This holds for all $w$ if and only if $(\ad D^a(v))^*v=-2a\tau v$. Since $\g$ is nilpotent, the operator $\ad D^a(v)$ and its transpose are nilpotent, so $a=0$ and $(\ad D^a(v))^*v=0$. Therefore, $\xi=v$, $b=D^a(v)$ and $(\ad b)^*v=0$, showing that $\phi$ takes the form~\eqref{eqn:acphi} and $\xi$ satisfies~\eqref{eqn:acxi}.
	Evaluating~\eqref{eqn:gtildealmostcontact} on $w,u$ gives
	\begin{multline*}
		g(w,u)-g(w,\xi)g(u,\xi)=\tilde g(\phi (w),\phi (u))\\
		=g(\frac12 (\ad w)^*\xi+\tau g(b,w)e_0,
		\frac12 (\ad u)^*\xi+\tau g(b,u)e_0)\\
		=\frac14g((\ad w)^*\xi,(\ad u)^*(\xi))+\tau g(b,w)g(b,u),
	\end{multline*}
	proving~\eqref{eqn:acg}.
	
	Lastly, evaluating~\eqref{eqn:gtildealmostcontact} on $e_0,e_0$ we get
	\[
	\tau =\tilde g(e_0,e_0)-\eta(e_0)\eta(e_0)=\tilde g(-b,-b)=g(b,b);\]
	however, this is a redundant condition, for $g(b,\xi)=g(D^a(\xi),\xi)=0$, so~\eqref{eqn:acg} and~\eqref{eqn:acxi} imply
	$g(b,u)=\tau g(b,b)g(b,u)$ for all $u$, which is equivalent to $g(b,b)=\tau$.
	
	The converse is proved in the same way.
\end{proof}

Now observe that we can write
\[g((\ad w)^*(v),u)=g(v,[w,u])=-dv^\flat(w,u)=-g((w\hook d v^\flat)^\sharp,u),\]
so $(\ad w)^*(\xi) = -(w\hook d\eta)^\sharp$. Recall that $d$ denotes the Chevalley-Eilenberg operator on $\g$, not $\tilde\g$.

\begin{lemma}
	\label{lemma:derivativeoftwoform}
	Let $g$ be a metric on a Lie algebra $\g$. Let $\Phi$ be a $2$-form. Then
	\[\nabla_x \Phi=\frac12\Lie_x\Phi-\frac12(\ad x)^*\Phi+\frac12\alpha^\Phi_x,\]
	where
	\[\alpha^\Phi_x(u,w)= \Phi(\ad (u)^*(x),w)-\Phi(\ad (w)^*(x),u).\]
\end{lemma}
\begin{proof}
	Using~\eqref{eqn:lc} we have:
	\begin{multline*}
		\nabla_x\Phi(u,w)=-\Phi(\nabla_x u,w)-\Phi(u,\nabla_x w)\\
		=\frac12\bigr(\Phi((\ad x)^*u+(\ad u)x+(\ad u)^*x,w)-\Phi((\ad x)^*w+(\ad w)x+(\ad w)^*x,u)\bigl)\\
		=-\frac12(\ad x)^*\Phi(u,w) -\frac12\Phi(\Lie_x u,w)+\frac12\Phi(\Lie_x w,u)
		+\frac12\alpha^\Phi_x(u,w)\\
		=-\frac12(\ad x)^*\Phi(u,w) +\frac12\Lie_x\Phi( u,w)
		+\frac12\alpha^\Phi_x(u,w).\qedhere
	\end{multline*}
\end{proof}

\begin{proposition}
	\label{prop:sasakistandard}
	Let $\g$ be a nilpotent Lie algebra with a pseudo-Riemannian metric $g$, let $D$ be a derivation and $\tau=\pm1$. Then $\tilde\g=\g\rtimes_D\Span{e_0}$ has a Sasaki structure $(\phi,\xi,\eta,\tilde g)$ such that
	$\tilde g=g+\tau e^0\otimes e^0$
	if and only if for some $\xi\in\g$, $b=D^a(\xi)$, $\eta=\xi^\flat$, writing
	\[\alpha_x(u,w)= d\eta(\ad (u)^*(x),w)
	-d\eta(\ad (w)^*(x),u),\]
	the following hold for $x,y\in\g$:
	\begin{gather}
		\label{eqn:sasakixi}
		D(\xi)=0, \qquad (\ad \xi)^s=0,\qquad (\ad b)^*(\xi)=0,\\
		\label{eqn:sasakiD}
		D^a(d\eta)=0, \qquad D^a(b)=-\tau\xi,\\
		\label{eqn:sasakietax}
		\eta\wedge x^\flat =\frac14\alpha_x -\frac14(\ad x)^*(d\eta)
		+\frac14d(\Lie_x\eta)+\tau b^\flat \wedge D^s(x)^\flat,\\
		\label{eqn:sasakib}
		D^s(x)\hook d\eta + x\hook db^\flat + b\hook dx^\flat + [x,b]^\flat=0.
	\end{gather}
	Then $\phi$ is given by
	\[\phi(w)=\frac12 (\ad w)^*(\xi)+\tau g(b,w)e_0, \qquad \phi(e_0)=-b, \quad w\in\g.\]
\end{proposition}
\begin{proof}
	Suppose $(\phi,\xi,\eta,\tilde g)$ is a Sasaki structure as in the hypothesis. Since Sasaki structures satisfy $\widetilde\nabla_X\xi=-\phi(X)$, by Lemma~\ref{lemma:acstandard} equations~\eqref{eqn:acphi}, \eqref{eqn:acxi},~\eqref{eqn:acg} hold. By Proposition~\ref{pr:SasakiEquivalent}, the Sasaki condition implies
	\begin{equation}
		\label{eqn:sasakibyform}
		\eta\wedge X^\flat=\widetilde\nabla_X\Phi.
	\end{equation}
	We have
	\begin{gather*}
		\Phi(u,w)=\tilde g(u,\phi(w))=\frac12 g(u,(\ad w)^*(\xi))=-\frac12 g([u,w],\xi),\\
		\Phi(e_0,w)=\tilde g(e_0,\phi(w))=g(b,w).
	\end{gather*}
	Thus,~\eqref{eqn:sasakibyform} for $X=e_0$ implies
	\begin{multline*}
		0=(\widetilde\nabla_{e_0}\Phi)(u,w)=-\Phi(\widetilde\nabla_{e_0}u,w)-\Phi(u,\widetilde\nabla_{e_0}w)
		=-\Phi(D^a(u),w)-\Phi(u,D^a(w))\\
		=\frac12 g([D^a(u),w],\xi)+\frac12 g([u,D^a(w)],\xi)=-\frac12d\eta(D^a(u),w)
		-\frac12d\eta(u,D^a(w))\\
		=\frac12(D^a d\eta)(u,w).
	\end{multline*}
	Similarly,
	\begin{multline*}
		-\tau g(w,\xi)
		=(\widetilde\nabla_{e_0}\Phi)(e_0,w)=-\Phi(e_0,\widetilde\nabla_{e_0} w)\\
		=-\Phi(e_0,D^a(w))=-g(b,D^a(w))=g(D^a(b),w),
	\end{multline*}
	i.e. $D^a(b)=-\tau \xi$.
	
	Then,~\eqref{eqn:sasakibyform} for $X=x\in\g$ gives
	\begin{align*}
		&g(u,\xi)g(x,w)-g(x,u)g(\xi,w)=(\widetilde\nabla_{x}\Phi)(u,w)=-\Phi(\widetilde\nabla_xu,w)-\Phi(u,\widetilde\nabla_x w)\\
		&=\Phi (\ad(u)^s(x)+\frac12 (\ad x)^*(u)-\tau g(D^s(u),x)e_0,w)\\
		&\quad-\Phi (\ad(w)^s(x)+\frac12 (\ad x)^*(w)-\tau g(D^s(w),x)e_0,u)\\
		&=-\frac12g\biggl([\ad(u)^s(x)+\frac12 (\ad x)^*(u),w]-[\ad(w)^s(x)+\frac12 (\ad x)^*(w),u],\xi\biggr)\\
		&\quad-\tau g(b,w)g(D^s(x),u)+\tau g(b,u)g(D^s(x),w)\\
		&=-\frac14g\biggl(\bigl[[u,x]+(\ad u)^*x+(\ad x)^*u,w\bigr]
		-\bigl[[w,x]+(\ad w)^*x+(\ad x)^*w,u\bigr],\xi\biggr)\\
		&\quad+\tau(b^\flat\wedge D^s(x)^\flat)(u,w)\\
		&=-\frac14g\biggl(\bigl[(\ad u)^*x+(\ad x)^*u,w\bigr]
		-\bigl[(\ad w)^*x+(\ad x)^*w,u\bigr]+[[u,w],x],\xi\biggr)\\
		&\quad+\tau (b^\flat\wedge D^s(x)^\flat)(u,w)\\
		&=\frac14d\eta(\ad (u)^*x+(\ad x)^*u,w)
		-\frac14d\eta(\ad (w)^*x+(\ad x)^*w,u)\\
		&\quad-\frac14d\eta(x,[u,w])+\tau(b^\flat\wedge D^s(x)^\flat)(u,w)\\
		&=\frac14\alpha_x(u,w)-\frac14(\ad x)^*(d\eta)(u,w)
		+\frac14d(\Lie_x\eta)(u,w)
		+\tau (b^\flat \wedge D^s(x)^\flat) (u,w)
	\end{align*}
	so
	\[\eta\wedge x^\flat = \frac14\alpha_x-\frac14(\ad x)^*(d\eta)
	+\frac14d(\Lie_x\eta)+\tau(b^\flat \wedge D^s(x)^\flat).\]
	Finally,
	\begin{multline*}
		0=(\widetilde\nabla_{x}\Phi)(e_0,w)=
		-\Phi(\widetilde\nabla_xe_0,w)-\Phi(e_0,\widetilde\nabla_x w)
		=\Phi(D^s(x),w)-\Phi(e_0,\nabla_x w)\\
		=\frac12 g([w,D^s(x)],\xi)
		-g(b,\nabla_x w)\\
		=\frac12 g(D^s(x),(\ad w)^*(\xi))
		+g(b,\ad (w)^s(x)+ \frac12 (\ad x)^*(w))\\
		=-\frac12 d\eta(w,D^s(x))
		+\frac12g\bigl(b,\ad (w)(x)+(\ad w)^*(x)+ (\ad x)^*(w)\bigr).
	\end{multline*}
	Equivalently,
	\begin{multline*}
		0=-d\eta(w,D^s(x))+g(b,\ad (w)(x)+(\ad w)^*(x)+ (\ad x)^*(w))\\
		=-d\eta(w,D^s(x))+db^\flat(x,w)+dx^\flat(b,w)+g([x,b],w)\\
		=(D^s(x)\hook d\eta + x\hook db^\flat + b\hook dx^\flat + [x,b]^\flat)(w).
	\end{multline*}
	
	Conversely, define $(\phi,\xi,\eta,\tilde g)$ as in the statement, and assume that~\eqref{eqn:sasakixi}--\eqref{eqn:sasakib} hold. Since $\ad\xi$ is antisymmetric,
	\[\ad\xi=-(\ad \xi)^*,\qquad \xi\hook d\eta = -(\ad \xi)^*(\xi)^\flat = (\ad\xi)(\xi)^\flat=0.\]
	Evaluating~\eqref{eqn:sasakietax} on $u,\xi$, one obtains
	\begin{align*}
		&\quad g(u,\xi)g(x,\xi)-g(x,u)\\
		&=\frac14d\eta(\ad (u)^*x+(\ad x)^*u,\xi)
		-\frac14d\eta(\ad (\xi)^*x+(\ad x)^*\xi,u)\\
		&\quad -\frac14d\eta(x,[u,\xi])+\tau (b^\flat\wedge D^s(x)^\flat)(u,\xi)\\
		&=-\frac14d\eta(-[\xi,x],u)-\frac14d\eta(x,[u,\xi])\\
		&\quad-\frac14d\eta((\ad x)^*\xi,u)+\tau g(b,u)g(D^s(x),\xi)\\
		&=-\frac14\eta([\xi ,[u,\xi]])+\frac14(u\hook d\eta)((\ad x)^*\xi)+\tau g(b,u)g(x,D^s\xi)\\
		&=-\frac14g((\ad u)^*\xi,(\ad x)^*\xi)-\tau g(b,u)g(x,b),
	\end{align*}
	which is equivalent to~\eqref{eqn:acg}. Since~\eqref{eqn:sasakixi} is assumed to hold and $\phi$ is defined so as to satisfy~\eqref{eqn:acphi}, Lemma~\ref{lemma:acstandard} implies that $(\phi,\xi,\eta,\tilde g)$ is an almost contact metric structure. In order to prove that it is Sasaki, one only needs to verify that~\eqref{eqn:sasakibyform} holds, which follows from the computations above.
\end{proof}

\begin{remark}
	The $2$-form $\alpha_x$ of Proposition~\ref{prop:sasakistandard} corresponds to the $2$-form $\alpha^{\Phi}_x$ of Lemma~\ref{lemma:derivativeoftwoform} with $\Phi$ equal to $d\eta$.
\end{remark}

\begin{remark}
	Using Lemma~\ref{lemma:derivativeoftwoform}, we see that~\eqref{eqn:sasakietax} can be rewritten as
	\begin{equation}
		\label{eqn:sasakietaxnabla}
		\eta\wedge x^\flat = \frac12\nabla_x d\eta + \tau b^\flat\wedge D^s(x)^\flat.
	\end{equation}
	Using equation~\eqref{eqn:lc}, we can read condition~\eqref{eqn:sasakib} as:
	\[
	D^s(x)\hook d\eta=\nabla_x b.
	\]
\end{remark}

\begin{remark}
	It is well known that on a Sasaki Lie algebra $\tilde\g$ the center is contained in $\Span{\xi}$; indeed, any element of the center satisfies $v\hook d\eta=0$, so it is a multiple of $\xi$.
	
	If $\tilde\g$ has nontrivial center, then $\lie z(\tilde\g)=\Span{\xi}$ and the quotient $\check\g=\g/\Span{\xi}$ has an induced pseudo-K\"ahler structure $(\hat g, J,\omega)$ by Proposition~\ref{prop:kahlerquotient}.
\end{remark}

\begin{remark}
	The equations of Proposition~\ref{prop:sasakistandard} simplify if we assume that the center is nontrivial, because then $\ad\xi=0$. However, the center may be trivial on a Sasaki Lie algebra, see e.g. Example~\ref{example:4.3}. It is noteworthy that Example~\ref{example:4.3} is isometric to a standard Lie algebra with nontrivial center (see Example~\ref{ex:43nontrivialcenter}).
\end{remark}

\section{$\lie z$-Standard Sasaki structures}\label{sec:SasakiReduction}
In this section we study the particular case where the vector $b$ of Proposition~\ref{prop:sasakistandard} is central in $\lie g$. More precisely, we say that a Sasaki structure $(\tilde\phi,\tilde\xi,\tilde\eta,\tilde g)$ on a Lie algebra $\tilde\g$ is \emph{$\lie z$-standard} if there is a standard decomposition $\tilde\g=\g\rtimes_D\Span{e_0}$ with $b=-\phi(e_0)$ in the center of $\g$ and $\tilde g=g+\tau e^0\otimes e^0$, with $\tau=\pm1$.

We will start by giving a geometric interpretation of this condition; to that end, we will need to recall a well-known construction. Let $\tilde\g$ be a Lie algebra with a Sasaki structure $(\tilde\xi,\tilde\eta,\tilde g,\tilde\phi)$. Let $X$ be a nonzero vector in $\tilde\g$. The associated, left-invariant Sasaki structure on the connected, simply connected group $\tilde G$ with Lie algebra $\tilde\g$ is invariant under the left action of the group $\{\exp tX\}$. The fundamental vector field $X^*$ is defined by
\[X^*_{g}=\frac{d}{dt} ( \exp t X)g,\]
so identifying $T_g\tilde G$ with $\tilde\g$ by left-translation we get
\[L_{g^{-1}*}X^*_{g}=\frac{d}{dt} g^{-1}( \exp t X)g=\Ad(g^{-1})X.\]
The moment map $\mu\colon\tilde G\to\R$ is by definition
\[\mu (g)=\eta(\Ad(g^{-1})X).\]
Therefore,
\begin{multline*}
	d\mu_g(L_{g*}v)=\frac{d}{dt}|_{t=0} \mu (g\exp tv)\\
	=\frac{d}{dt}|_{t=0}\eta(\Ad(\exp -tv)\Ad(g^{-1})X)
	=-\eta([v,\Ad(g^{-1})X]).
\end{multline*}
Now if $\mu(g)=0$, $\Ad(g^{-1})X\in\ker\eta$. This implies that $\Ad(g^{-1})X\hook d\eta$ is nonzero, i.e. there is some $v$ such that $\eta([v,\Ad(g^{-1})X])\neq0$. Thus, $0$ is a regular value and $\mu^{-1}(0)$ is a hypersurface.

Since $X^*$ is nowhere zero, the action of $\{\exp tX\}$ is well defined on $\mu^{-1}(0)$. Therefore, the quotient
\[\tilde G//\{\exp tX\}=\mu^{-1}(0)/\{\exp tX\}\]
is well defined (locally), and it has an induced Sasaki structure.

$\lie z$-standard Sasaki structures can be characterized as follows:
\begin{lemma}
	\label{lemma:3littleconditions}
	Let $\tilde\g$ be a Lie algebra with a Sasaki structure $(\phi,\xi,\eta,\tilde g)$. The following are equivalent:
	\begin{enumerate}[label=(\textit{\roman*})]
		\item\label{cond:3little.1} there is a standard decomposition $\tilde\g=\g\rtimes_D\Span{e_0}$ with $\phi(e_0)$ in the center of $\g$;
		\item\label{cond:3little.2} $\tilde\g$ contains a vector $X$ with $\tilde g(X,X)\neq0$ such that its centralizer $\lie z(X)$ is a nilpotent ideal of codimension one;
		\item\label{cond:3little.3} the simply connected Lie group $\tilde G$ with Lie algebra $\tilde\g$ has a one-parameter subgroup $\{\exp tX\}$ such that
		\begin{itemize}
			\item $\tilde g(X,X)\neq0$;
			\item the zero set of the moment map is a normal nilpotent subgroup $G$; and
			\item $\{\exp{tX}\}$ commutes with $G$.
		\end{itemize}
	\end{enumerate}
\end{lemma}
\begin{proof}
	If~\ref{cond:3little.1} holds, observe that $e_0$ is not a multiple of $\xi$ by Proposition~\ref{prop:sasakistandard}; thus, $X=-\phi(e_0)$ has centralizer equal to $\g$. This implies ~\ref{cond:3little.2}.
	
	Now assume that~\ref{cond:3little.2} holds; then $\tilde\g$ is solvable, as it contains a codimension one nilpotent ideal. The zero level set of the moment map $\{g\mid \eta(\Ad(g^{-1})X)=0\}$ is the connected subgroup with Lie algebra $\lie z(X)$, giving~\ref{cond:3little.3}.
	
	Finally, suppose that~\ref{cond:3little.3} holds. Since $\mu^{-1}(0)$ is a normal nilpotent subgroup, its Lie algebra is the nilpotent ideal
	\[\g=\ker X\hook d\eta.\]
	In addition, $\mu^{-1}(0)$ contains the identity, so $\eta(X)=0$. This implies that $\g$ has codimension one. By construction, $e_0=\phi(X)$ is orthogonal to $\g$. Since $X$ is not lightlike, the restriction of the metric to $\g$ is definite; hence we have a standard decomposition $\tilde\g=\g\rtimes\Span{e_0}$. By construction, $\phi(e_0)=-X$, so it is central in $\g$, giving~\ref{cond:3little.1}.
\end{proof}
Given a $\lie z$-standard Sasaki structure, Lemma~\ref{lemma:3littleconditions} implies that $\{\exp tX\}$ is central in $G$, so the right action of $\{\exp tX\}$ preserves the Sasaki structure and the
quotient $G/\exp\{tX\}$ is a Lie group with Lie algebra $\lie{z}(X)/\Span{X}$, which is Sasaki by construction. Conversely, we can express $\lie z(X)$ as a central extension of $X$, and then express $\g$ as a standard extension of $\lie z(X)$.

\begin{example}
	In Example~\ref{ex:43nontrivialcenter}, $\{\exp te_2\}$ satisfies the conditions of Lemma~\ref{lemma:3littleconditions}; the three-dimensional quotient in this case is the Heisenberg algebra, with its Sasaki structure.
\end{example}

In the language of Proposition~\ref{prop:sasakistandard}, we can express this as follows:
\begin{corollary}
	\label{cor:sasakistandardcentral}
	Let $\g$ be a nilpotent Lie algebra with a pseudo-Riemannian metric $g$, $D$ a derivation and $\tau=\pm1$. Assume $\tilde\g=\g\rtimes_D\Span{e_0}$ has a $\lie z$-standard Sasaki structure $(\phi,\xi,\eta,\tilde g)$. Then
	the following hold for $x\in\g$:
	\begin{gather*}
		D(\xi)=0, \qquad D(b)=-2\tau \xi+hb, \quad h\in\R, \quad b,\xi\in\lie z(\g)\\
		D^a(d\eta)=0, \qquad D(d\eta)=2db^\flat,\\
		\eta\wedge x^\flat =\frac12\nabla_x d\eta+\tau b^\flat \wedge D^s(x)^\flat,\\
		d\eta(D^s(x),y)=d\eta(x,D^s(y)).
	\end{gather*}
	Furthermore, $\phi$ is given by
	\[\phi(w)=\frac12 (\ad w)^*(\xi)+\tau g(b,w)e_0, \qquad \phi(e_0)=-b, \quad w\in\g.\]
	In addition, $\g/\Span{b}$ has a Sasaki structure $(\check\phi,\check\xi,\check\eta,\check g)$ induced by the identification $\Span{e_0,b}^\perp \cong \g/\Span{b}$; at the level of the corresponding Lie groups, this amounts to taking the Sasaki reduction by the left action of the one-parameter subgroup $\{\exp tb\}$.
\end{corollary}
\begin{proof}
	We specialize Proposition~\ref{prop:sasakistandard} with $b=-\phi(e_0)$ central. Then $(\ad b)^*$ and $b\hook dx^\flat$ are zero. In particular, from~\eqref{eqn:sasakib}, we get
	\begin{equation}
		\label{eqn:Dsxhooketa}
		D^s(x)\hook d\eta + x\hook db^\flat=0.
	\end{equation}
	For $x=b$, this implies $D^s(b)\hook d\eta=0$. Since $d\eta$ is nondegenerate on $\Span{b,\xi}^\perp\!$, this implies that $D^s(b)\in\Span{b,\xi}$. Furthermore, we have
	\[g(D^s(b),\xi)=g(b,D^s(\xi))=g(b,-b)=-\tau,\]
	so $D^s(b)=-\tau\xi + hb$ for some real constant $h$. Therefore,
	\[D(b)=-2\tau\xi+hb.\]
	Since $D$ is a derivation, we have
	\[0=D[b,x]=[D(b),x]+[b,D(x)]=-2\tau[\xi,x].\]
	Therefore $\xi$ is in the center of $\g$.
	
	By~\eqref{eqn:sasakiD}, $D^a(d\eta)=0$, so we observe that
	\begin{multline}
		D^sd\eta(x,y)=Dd\eta(x,y)=-d\eta(Dx,y)-d\eta(x,Dy)\\
		=\eta([Dx,y]+[x,Dy])=\eta(D[x,y])=-2g(b,[x,y])=2db^\flat(x,y).
	\end{multline}
	Therefore, $D(d\eta)=2db^\flat$ and~\eqref{eqn:Dsxhooketa} becomes equivalent to
	\[0=d\eta(D^s(x),y)+\frac12(D^sd\eta)(x,y)=\frac12\bigl(d\eta(D^s(x),y)-d\eta(x,D^s(y))\bigr).\]

	For the last part, observe that $\lie g$ is the centralizer of $b$ in $\tilde\g$, and apply the observation before the statement. The fact that $(\check\phi,\check\xi,\check\eta,\check g)$ is Sasaki can be seen from $\eta\wedge x^\flat=\frac12\check\nabla_xd\eta$.
\end{proof}

We can describe the situation of Corollary~\ref{cor:sasakistandardcentral} in terms of the K\"ahler quotient as follows:
\begin{corollary}
	\label{cor:reductiontokahler}
	Let $\g$ be a nilpotent Lie algebra with a pseudo-Riemannian metric $g$, $D$ a derivation and $\tau=\pm1$. Assume $\tilde\g=\g\rtimes_D\Span{e_0}$ has a $\lie z$-standard Sasaki structure $(\phi,\xi,\eta,\tilde g)$. Then $\xi$ is central in $\g$ and there is $h\in\R$ such that
	\begin{enumerate}
		\item $g(\xi,\xi)=1$, $g(b,b)=\tau$, $g(b,\xi)=0$;
		\item the quotient $\check \g=\g/\Span{b,\xi}$ has a pseudo-K\"ahler structure $(\check g,J,\omega)$ with $(\g,g)\to(\check \g,\check g)$ a Riemannian submersion, $\omega=\frac12d\eta$ and $\check D(\omega)=db^\flat$;
		\item relative to the splitting $\Span{b,\xi}^\perp\oplus\Span{b}\oplus\Span{\xi}$, $D$ takes the form
		\[
		D=\begin{pmatrix} \check D & 0 & 0 \\ 0 & h & 0 \\ 0 & -2\tau & 0\end{pmatrix};\]
		\item $[J,\check D]=0$;
		\item $\check D$ is a derivation and $[\check D^s,\check D^a]=h\check D^s-2(\check D^s)^2$.
	\end{enumerate}
\end{corollary}
\begin{proof}
	Define $b=\phi(e_0)$, hence $g(\xi,\xi)=1$ by definition of Sasaki and
	\[g(b,\xi)=\tilde g(b,\xi)=-\tilde g(e_0,\phi(\xi))=0,\qquad g(b,b)=\tilde{g}(e_0,e_0)=\tau\]
	give the first condition.
	
	Let $\check\g=\g/\Span{b,\xi}$. Then arguing as in Proposition~\ref{prop:kahlerquotient} we see that $\check\nabla d\eta$ is the projection of $\nabla d\eta$; projecting the equation~\eqref{eqn:sasakietaxnabla},
	we see that $d\eta$ is $\check\nabla$-parallel. Furthermore, for $x$ orthogonal to $b,\xi$, we get by taking the interior product of~\eqref{eqn:sasakietaxnabla} with $\xi$ that
	\[x^\flat = \frac12\xi\hook \nabla_x d\eta - g(D^s(x),\xi)\tau b^\flat= \frac12\xi\hook \nabla_x d\eta;\]
	using Lemma~\ref{lemma:derivativeoftwoform}, we get
	\begin{equation}
		\label{eqn:xflatequalsxihookdeta}
		x^\flat = \frac14\xi\hook (\alpha_x-(\ad x)^*d\eta +\Lie_x d\eta)
		=\frac14 (\ad x)^*\xi\hook d\eta.
	\end{equation}
	This implies that $d\eta$ is nondegenerate. Now set
	\[J(x)=-\frac12(x\hook d\eta)^\sharp.\]
	Then in $\Span{b,\xi}^\perp$ equation~\eqref{eqn:xflatequalsxihookdeta} reads
	\[x^\flat=-\frac14 (x\hook d\eta)^\sharp \hook d\eta = \frac12J(x)\hook d\eta=-\big(J\circ J(x)\big)^\flat=-\big(J^2(x)\big)^\flat;\]
	therefore, $J$ is an almost complex structure, and $(\check g,J,d\eta)$ is a pseudo-K\"ahler structure. In particular, we can write
	\[d\eta(x,y)=2g(x,Jy).\]
	Now from Corollary~\ref{cor:sasakistandardcentral} write
	\[d\eta(D^s(x),y)=d\eta(x,D^s(y))\]
	as
	\[g(JD^s(x),y)=g(Jx,D^s(y))=-g(x,JD^s(y)),\]
	i.e. $JD^s=-(JD^s)^*=D^sJ$. In addition, $D^ad\eta=0$ can be rewritten as
	\begin{multline*}
		0=D^ad\eta(x,y)=d\eta(D^ax,y)+d\eta(x,D^ay)\\
		=2g(D^ax,JY)+2g(x,JD^ay)=2g(x,[J,D^a]y).
	\end{multline*}
	This shows that $J$ and $D$ commute.
	
	The Lie bracket on $\check \g$ and the Lie bracket on $\g$ are related by
	\[[x,y]=[x,y]_{\check \g}-\tau db^\flat(x,y)b-d\eta(x,y)\xi;\]
	$b,\xi$ are in the center for $\g$. Relative to the splitting $\Span{b,\xi}^\perp\oplus\Span{b}\oplus\Span{\xi}$, $D$ takes the form
	\begin{equation}
		\label{eqn:DforReduction}
		D=\begin{pmatrix}\check D & 0 & 0 \\ 0 & h & 0 \\ 0 & -2\tau & 0\end{pmatrix}.
	\end{equation}
	A linear map $D$ of the form~\eqref{eqn:DforReduction} automatically satisfies $D[x,y]=[Dx,y]+[x,Dy]$ when $x$ lies in $\Span{b,\xi}$; therefore, $D$ is a derivation if and only if for $x,y$ in $\Span{b,\xi}^\perp$ one has
	\begin{align*}
		0=D[x,y]-[Dx,y]-[x,Dy]=&\check D[x,y]_{\check\g}-\tau db^\flat(x,y)(hb-2\tau\xi)\\
		&-[\check Dx,y]_{\check\g}+\tau db^\flat(\check Dx,y)b+d\eta(\check Dx,y)\xi\\
		&-[x,\check Dy]_{\check\g}+\tau db^\flat(x,\check Dy)b+d\eta(x,\check Dy)\xi.
	\end{align*}
	Thus, $D$ is a derivation if and only if $\check D$ is a derivation of $\check \g$ and
	\begin{gather*}
		h db^\flat(x,y) = db^\flat(\check Dx,y)+ db^\flat(x,\check Dy),\\
		-2 db^\flat(x,y)=d\eta(\check Dx,y)+d\eta(x,\check Dy),
	\end{gather*}
	where the latter is again $2 db^\flat=\check Dd\eta$.
	
	Then using $[J,D]=0$,
	\begin{multline*}
		db^\flat(x,y)=\frac12\check Dd\eta (x,y)=-\frac12d\eta(\check Dx,y)-\frac12d\eta(x,\check Dy)\\
		=-g(\check Dx,Jy)-g(x,J\check Dy)=-g(x, (\check D^*J+J\check D)y)=-2g(x,\check D^sJ y).
	\end{multline*}
	Thus
	\begin{multline*}
		2hg(x,\check D^sJy)=-hdb^\flat(x,y)=-db^\flat(\check Dx,y)- db^\flat(x,\check Dy)\\
		=2g(\check Dx,\check D^sJ y)+2g(x,\check D^sJ \check D y)
		=2g(x,(\check D^s-\check D^a)\check D^sJ y)+2g(x,\check D^s \check DJ y).
	\end{multline*}
	Therefore,
	\[h\check D^sJ=(\check D^s-\check D^a)\check D^sJ +\check D^s \check DJ\\
	=2(\check D^s)^2J+[\check D^s,\check D^a]J,
	\]
	i.e.
	\[h\check D^s-2(\check D^s)^2=[\check D^s,\check D^a].\qedhere\]
\end{proof}

In the situation of Corollary~\ref{cor:reductiontokahler}, we will say that the
pseudo-K\"ahler Lie algebra $\check\g$ is the \emph{K\"ahler reduction} of the
$\lie z$-standard Sasaki structure of $\tilde\g$. Notice that $\check\g$ is indeed a K\"ahler reduction in the sense of symplectic geometry, arising from the action of $\{\exp tb\}$ on the pseudo-K\"ahler nilmanifold $\tilde\g/\Span{\xi}$.

\begin{example}
	In Example~\ref{ex:43nontrivialcenter}, we have
	\begin{gather*}
		\check \g=\Span{e_3,e_4}, \qquad \check D=I, \qquad b=-e_2, \qquad h=2, \qquad \tau=-1,\\
		\omega=e^{34},\qquad db^\flat=de^2=-2e^{34}, \qquad d\eta=2e^{34}.
	\end{gather*}
\end{example}

Corollary~\ref{cor:sasakistandardcentral} has a K\"ahler analogue, which can be viewed as a consequence of Corollary~\ref{cor:reductiontokahler}, using the fact that any pseudo-K\"ahler Lie algebra yields a Sasaki Lie algebra by taking a central extension. Notice that this construction only works one way in general, i.e. it is not generally true that a Sasaki Lie algebra is a central extension of a pseudo-K\"ahler Lie algebra. This only occurs when $\xi$ is central, which happens to be true in the situation of Corollary~\ref{cor:reductiontokahler}.
\begin{proposition}
	\label{prop:kahlerreduction}
	Let $\g$ be a nilpotent Lie algebra with a pseudo-Riemannian metric $g$, let $D$ be a derivation and $\tau=\pm1$. Suppose that $\tilde\g=\g\rtimes_D\Span{e_0}$ has a pseudo-K\"ahler structure $(\tilde J,\tilde g,\tilde \omega)$ such that $\tilde g=g+\tau e^0\otimes e^0$, with $b=-\tilde Je_0$ in the center of $\g$. Then
	\begin{enumerate}
		\item the quotient $\check \g=\g/\Span{b}$ has a pseudo-K\"ahler structure $(\check g,\check J,\check \omega)$ with $\pi\colon(\g,g)\to(\check \g,\check g)$ a Riemannian submersion, $\pi^*\check\omega=\tilde\omega|_{\g}$ and $D(\omega)=db^\flat$;
		\item relative to the splitting $\Span{b}^\perp\oplus\Span{b}$, $D$ takes the form
		\[
		D=\begin{pmatrix} \check D & 0 \\ 0 & h \end{pmatrix};\]
		\item $[\check J,\check D]=0$;
		\item $\check D$ is a derivation and $[\check D^s,\check D^a]=h\check D^s-2(\check D^s)^2$.
	\end{enumerate}
\end{proposition}
\begin{proof}
	Write $\check\g=\Span{b}^\perp$ in $\g$, and let $\omega$ be the restriction of $\tilde\omega$ to $\check\g$. Then $\tilde\omega=\omega - \tau b\wedge e^0$.
	
	Let $\lie h=\g\oplus{\Span{\xi}}$ be the central extension of $\g$ by the cocycle $2\omega$, $\check{\lie h}$ the quotient $\lie h/\Span{b}$, and $\tilde{\lie h}$ the semidirect product $\lie h\rtimes_{D'}\Span{e_0}$, where $D'$ is defined by
	\[D'v=Dv, \quad v\in\check \g,\qquad D'\xi=0, \qquad D'b=Db-2\tau \xi.\]
	We can summarize the situation as follows
	\[\check{\lie h}=\check{\lie g}\oplus\Span{\xi}, \qquad \lie h=\check\g\oplus\Span{b,\xi}, \qquad \tilde{\lie h}=\check\g\oplus \Span{b,\xi,e_0}.\]
	We can view equivalently $\tilde{\lie h}$ as the central extension of $\tilde\g$ by $2\tilde\omega$. In particular, $\tilde{\lie h}$ has a Sasaki metric $(\tilde \phi, \xi, \tilde h,\tilde\eta)$ induced by the pseudo-K\"ahler metric of $\tilde\g$ (see~\cite{Hatakeyama:SomeNotes}). Explicitly, $\tilde\eta$ is the $1$-form on $\tilde{\lie h}$ that vanishes on $\tilde\g$, with $\tilde\eta(\xi)=1$, so that $d\eta=2\tilde\omega$, we have
	\[\tilde h =\tilde g+\tilde\eta\otimes \tilde\eta,\qquad \tilde \phi=\tilde J.\]
	Since $b$ is central in $\lie h$, we can apply Corollary~\ref{cor:reductiontokahler}. Then $(\check g, \check J,\check\omega)$ is pseudo-K\"ahler, and $\check D\omega=db^\flat$,
	\[
	D'=\begin{pmatrix} \check D & 0 & 0 \\ 0 & h & 0 \\ 0 & -2\tau & 0\end{pmatrix},\]
	proving items 1 and 2. Items 3 and 4 follow directly from Corollary~\ref{cor:reductiontokahler}.
\end{proof}

\section{Construction of $\lie z$-standard Sasaki structures}
In this section we invert the reduction process of Corollary~\ref{cor:reductiontokahler} and describe a constructive way of obtaining $\lie z$-standard Sasaki structures. We also classify $\lie z$-standard Sasaki structures of dimension $\leq 7$ whose K\"ahler reduction is abelian.

\begin{proposition}
	\label{prop:constructive}
	Let $(\check \g,J,\omega)$ be a pseudo-K\"ahler nilpotent Lie algebra. Let $\check D$ be a derivation of $\check \g$, $\tau=\pm1$, and $\g=\check\g\oplus\Span{b,\xi}$ a central extension of $\g$ with a metric of the form:
	\begin{gather*}
		g(x,y)=\check g(x,y), \qquad g(x,b)=0=g(x,\xi), \\
		g(\xi,\xi)=1, \qquad g(b,b)=\tau,\qquad g(b,\xi)=0,
	\end{gather*}
	where $x,y\in\check\g$.
	Assume furthermore
	\begin{itemize}
		\item $d\xi^\flat=2\omega$, where the right-hand-side is implicitly pulled back to $\g$;
		\item $db^\flat=\check D\omega$, where the right-hand-side is implicitly pulled back to $\g$;
		\item $[J,\check D]=0$;
		\item $[\check D^s,\check D^a]=h\check D^s-2(\check D^s)^2$ for some constant $h$.
	\end{itemize}
	Let $\tilde\g=\g\rtimes \Span{e_0}$, where
	\[[e_0,x]=\check Dx, \qquad [e_0,b]=hb-2\tau \xi, \qquad [e_0,\xi]=0;\]
	then $\tilde\g$ has a $\lie z$-standard Sasaki structure $(\phi,\eta,\xi,\tilde g)$ given by
	\[\tilde g=g+\tau e^0\otimes e^0, \qquad \phi(x)=J(x)+\tau g(b,x)e_0, \qquad \phi(e_0)=-b, \quad x\in\g.\]
\end{proposition}
\begin{proof}
	The fact that $D=\check D + \tau b^\flat\otimes (hb-2\tau\xi)$ is a derivation is proved as in Corollary~\ref{cor:reductiontokahler}.
	
	Then we use Proposition~\ref{prop:sasakistandard}. To prove~\eqref{eqn:sasakib}, write
	\begin{multline*}
		db^\flat(y,x)= \check D\omega(y,x)=-\omega(\check Dy,x)-\omega(y,\check Dx)=-g(\check Dy,Jx)-g(y,J\check Dx)\\
		=-g(y,(\check D^*J+J\check D)x)=-g(y,J(\check D+\check D^*)x)=-2\omega(y,\check D^sx)=-d\eta(y,\check D^sx);
	\end{multline*}
	then $D^s(x)\hook d\eta+x\hook db^\flat=0$, which is equivalent to~\eqref{eqn:sasakib} since $b$ is central.
	
	To prove~\eqref{eqn:sasakietaxnabla}, notice that projecting this equation to $\Lambda^2\check\g$ simply says that $\omega$ is parallel on $\check\g$. The interior product with $\xi$ yields~\eqref{eqn:xflatequalsxihookdeta}, which holds by construction. Finally, taking interior product of~\eqref{eqn:sasakietaxnabla} with $b$ and using the fact that $D^s(b)\in\Span{b,\xi}$, we compute
	\begin{multline*}
		0=\frac14b\hook (\alpha_x-(\ad x)^*d\eta +\Lie_x d\eta)+ D^s(x)^\flat\\
		=\frac14((\ad x)^*b\hook d\eta)+ D^s(x)^\flat =\bigl(\frac12 J((\ad x)^*b)+D^s(x))^\flat.
	\end{multline*}
	We also have $\ad(x)^*b=\ad (D^s(x))^*\xi=-2J(D^s(x))$. Therefore, this equation reduces to $J^2(D^s(x))=-D^s(x)$, which is automatically satisfied.
	
	The other hypotheses of Proposition~\ref{prop:sasakistandard} are trivially satisfied; therefore, $\tilde \g$ has a Sasaki structure with
	\[\phi(w)=\frac12(\ad w)^*\xi + \tau g(b,w)e_0=-w\hook \omega +\tau(g,b,w)e_0=Jw+ \tau(g,b,w)e_0.\qedhere\]
\end{proof}

\begin{remark}
	It is no loss of generality to assume $h\geq 0$; indeed, changing the sign of $\check D$, $e_0$, $b$ and $h$ gives the same Sasaki Lie algebra up to isometric isomorphism.
\end{remark}

\begin{remark}
	The hypotheses of Proposition~\ref{prop:constructive} are preserved if one rescales both $h$ and $\check D$. This yields different metrics on $\tilde \g$, which are however related by a
	$\mathcal{D}$-homothety (in particular, they have different curvature).
	
	Accordingly, one can assume that either $h=0$ or $h=2$ up to $\mathcal{D}$-homothety. The condition $h=0$ implies that $\Tr (\check D^s)^2=0$. If $\check\g$ is Riemannian, $\check D^s$ is diagonalizable, so $h=0$ implies that $\check D$ is skew-symmetric.
\end{remark}

\begin{remark}
	\label{remark:reversesign}
	One can always reverse the sign of the metric $\check g$ and the $2$-form $\omega$ and obtain a different Sasaki metric on an isomorphic Lie algebra $\tilde\g'$; the isomorphism is realized by the mapping $b\mapsto -b'$, $\xi\mapsto -\xi'$.
\end{remark}

Let $(\check \g_0,J_0,g_0,\omega_0)$, $(\check \g_1,J_1,g_1,\omega_1)$ be pseudo-K\"ahler Lie algebras, with $\g_1$ abelian. Let $\rho\colon\check\g_0\to\gl(\check\g_1)$ be a representation such that
\begin{equation}
	\label{eqn:pseudokahlersemidirect}
	\rho(X)\omega_1=0, \qquad [J_1,\rho(X)]+[\rho (J_0X),J_1]J_1=0.
\end{equation}
Then $\check \g_0\ltimes\check\g_1$ has an almost Hermitian structure $(g,J,\omega)$, with $g=g_0+g_1$, $\omega=\omega_0+\omega_1$, and $J=\begin{pmatrix} J_1 & 0 \\ 0 & J_2\end{pmatrix}$. It is straightforward to check that $\omega$ is closed and $J$ integrable, i.e. $\check \g_0\ltimes\check\g_1$ is pseudo-K\"ahler. In addition, the projection $\pi_1$ on the factor $\check\g_1$ is a derivation, giving a one-parameter family of derivations $\check D=\frac h2\pi_1$ that satisfy the hypotheses of Proposition~\ref{prop:constructive}. The resulting Sasaki extension $\tilde\g$ takes the form
\begin{equation}
	\label{eqn:gradedsasaki}
	\begin{split}
		(\check \g_0\ltimes\check\g_1\oplus\Span{b,\xi})\rtimes\Span{e_0}, \qquad d\xi^\flat=2\omega, \qquad db^\flat=-h \omega,\\
		[e_0,X_0]=0, \qquad [e_0,X_1]=\frac h2 X_1, \qquad [e_0,b]=hb-2\tau\xi, \qquad [e_0,\xi]=0,
	\end{split}
\end{equation}
where $X_0$ denotes the generic element of $\check \g_0$ and $X_1$ the generic element of $\check\g_1$.

\begin{proposition}
	\label{prop:canassumeDsymmetric}
	In the hypotheses of Proposition~\ref{prop:constructive}, if $\check D^s$ is a derivation and $[\check D^s,\check D^a]=0$, we can assume up to isometry that $\check\g$ is a semidirect product $\check\g=\check\g_0\ltimes_\rho\check\g_1$, where $\check\g_0, \check\g_1$ are pseudo-K\"ahler with $\check\g_1$ abelian, $\check D=\frac h2\pi_1$ and $\tilde\g$ takes the form~\eqref{eqn:gradedsasaki}.
\end{proposition}
\begin{proof}
	Write $\tilde\g=\g\rtimes\Span{e_0}$, where $\ad (e_0)=\check D + hb^*\otimes (hb-2\tau\xi)$.
	Then define
	\[\chi\colon\Span{e_0}\to\Der\g, \qquad \chi(e_0)=\check D^s+ hb^*\otimes (hb-2\tau\xi).\]
	Then $\chi(e_0)^s=\ad (e_0)^s$ and $[\chi(e_0),\ad e_0]=0$. Thus, the Lie algebra $\g\rtimes_\chi\Span{e_0}$ is isometric to the Lie algebra $\tilde\g$ constructed in Proposition~\ref{prop:constructive}. In other words, replacing $\check D$ with $\check D^s$ gives the same metric $\tilde g$ up to isometry. In addition, $\check D\omega=\check D^s\omega$, so $db^\flat$ is unchanged.
	
	By Proposition~\ref{prop:constructive}, the minimal polynomial of $\check D$ divides $p(t)=ht-2t^2$. Thus $\check D$ is diagonalizable over $\R$, and takes the form
	\[\begin{pmatrix} 0 & 0 \\ 0 & \frac h2 I \end{pmatrix}\]
	in some basis; since $\check D$ commutes with $J$, its eigenspaces are $J$-invariant. Since it is symmetric, they are orthogonal. Since a diagonalizable derivation defines a grading, we have $\check\g=\check\g_0\ltimes_\rho\check\g_1$, the K\"ahler form splits as $\omega_0+\omega_1$ and
	\[J=\begin{pmatrix}J_0 & 0 \\ 0 & J_1\end{pmatrix}.\]
	We have that $(\check \g_0,J_0,\omega_0)$ is K\"ahler, $\check\g_1$ is abelian, and~\eqref{eqn:pseudokahlersemidirect} holds.
\end{proof}

\begin{corollary}
	\label{cor:diagonalizable}
	In the hypotheses of Proposition~\ref{prop:constructive}, if $\check D^s$ is a derivation and it is diagonalizable over $\C$, then we can assume up to isometry that $\check\g$ is a semidirect product $\check\g=\check\g_0\ltimes_\rho\check\g_1$, where $\check\g_0, \check\g_1$ are pseudo-K\"ahler with $\check\g_1$ abelian, $\check D=\frac h2\pi_1$ and $\tilde\g$ takes the form~\eqref{eqn:gradedsasaki}.
\end{corollary}
\begin{proof}
	Denote by $\check\g^\C$ the complexification of $\check\g$, with the scalar product obtained by complexifying the scalar product of $\check\g$. The complexified endomorphisms $(\check D^s)^\C\colon\check\g^\C\to\check\g^\C$, $(\check D^a)^\C\colon\check\g^\C\to\check\g^\C$ are symmetric and antisymmetric, respectively. Furthermore, we get
	\begin{equation}
		\label{eqn:complexified}
		[(\check D^s)^\C, (\check D^a)^\C]=h(\check D^s)^\C-2((\check D^s)^\C)^2.
	\end{equation}
	By hypothesis, there exists an orthonormal basis of eigenvectors of $(\check D^s)^\C$. Then $(\check D^s)^\C$ is diagonal in this basis, and $(\check D^a)^\C$ has zero on the diagonal. Therefore, $[(\check D^s)^\C, (\check D^a)^\C]$ has zero on the diagonal, so~\eqref{eqn:complexified} implies that it vanishes and we can apply Proposition~\ref{prop:canassumeDsymmetric}.
\end{proof}
In particular, Corollary~\ref{cor:diagonalizable} classifies $\lie z$-standard Sasaki structures that reduce to an abelian K\"ahler Lie algebra, as positive-definiteness of the metric implies that $\check D^s$ is automatically a diagonalizable derivation in this case.

The case of indefinite signature is more flexible, as we will see below. Notice that the signature of a pseudo-K\"ahler metric is necessarily of the form $(2p,2q)$.
\begin{theorem}
	\label{thm:central5}
	Let $\tilde\g$ be a Lie algebra of dimension $5$ with a $\lie z$-standard Sasaki structure. Then, up to isometry and $\mathcal{D}$-homothety, $\tilde\g$ is one of
	\begin{gather*}
		(0,0,0,-2e^{12}-2\tau e^{35},0),\\
		(0,0,2e^{35},-2e^{12}-2\tau e^{35},0),\\
		(e^{15},e^{25},2\tau e^{12}+2e^{35},-2e^{12}-2\tau e^{35},0),
	\end{gather*}
	and the Sasaki structure is given by
	\[\tilde g=\pm(e^1\otimes e^1+e^2\otimes e^2)+\tau e^3\otimes e^3+e^4\otimes e^4+\tau e^5\otimes e^5,
	\quad \xi=e_4, \quad \Phi=-e^{12}-\tau e^{35}.\]
\end{theorem}
\begin{proof}
	The K\"ahler reduction $\check\g$ is a nilpotent Lie algebra of dimension two, hence abelian. Assume first that $\check\g$ has positive-definite signature. In some basis $\{e_1,e_2\}$, we have
	\[\check g=e^1\otimes e^1+e^2\otimes e^2, \qquad \omega=-e^{12}, \qquad J=e^1\otimes e_2 - e^2\otimes e_1.\]
	Derivations that commute with $J$ lie in $\Span{I,J}$. In particular, $\check D^s$ commutes with $\check D^a$, so Proposition~\ref{prop:canassumeDsymmetric} implies that up to isometry we can assume $\check D=0$ or $\check D=\frac h2I$.
	
	Up to $\mathcal{D}$-homothety, we can assume that either $h=0$ or $h=2$.
	
	For $h=0$,~\eqref{eqn:gradedsasaki} gives
	\[\tilde\g=(0,0,0,-2e^{12}-2\tau e^{35},0);\]
	for $h=2$, either $\check D=0$ and
	\[\tilde\g=(0,0,2e^{35},-2e^{12}-2\tau e^{35},0),\]
	or $\check D=I$ and
	\[\tilde\g=(e^{15},e^{25},2\tau e^{12}+2e^{35},-2e^{12}-2\tau e^{35},0).\]
	
	In either case, the metric is
	\[\tilde g=e^1\otimes e^1+e^2\otimes e^2+\tau e^3\otimes e^3+e^4\otimes e^4+\tau e^5\otimes e^5.\]
	Taking into consideration the negative-definite metric on $\check\g$ has the effect of adding the $\pm$ signs, as per Remark~\ref{remark:reversesign}.
\end{proof}
Notice that the third Lie algebra appearing in Theorem~\ref{thm:central5} is Example~\ref{ex:43nontrivialcenter}.

We proceed to give a list of the
$7$-dimensional Lie algebras with a $\lie z$-standard Sasaki structure that reduces to an abelian pseudo-K\"ahler Lie algebra $\check\g$ up to isometry and $\mathcal{D}$-homothety. This list is given in Table~\ref{table:ZstdSasaki7D}, where we write the diagonal metric $\tilde{g}$ as a line vector with respect to the basis $\{e^1,\dots,e^7\}$, using the convention that $[1]_n$ is a vector of $n$ elements, each equal to $1$. For example $[1]_4=(1,1,1,1)$ and $(\pm [1]_4, \tau, +1, \tau )$ represents the metric
\[\tilde g=\pm(e^1\otimes e^1+e^2\otimes e^2+e^3\otimes e^3+e^4\otimes e^4)+\tau e^5\otimes e^5+e^6\otimes e^6+\tau e^7\otimes e^7.\]

\begin{table}[th]
	\centering
	\caption{$7$-dimensional Lie algebras with a $\lie z$-standard Sasaki structure that reduces to an abelian pseudo-K\"ahler Lie algebra $\check\g$ up to isometry and $\mathcal{D}$-homothety\label{table:ZstdSasaki7D}}
	{\footnotesize
		\begin{tabular}{l C C}
			\toprule
			n. & \tilde{\g} & \textnormal{Metric } \tilde{g} \\
			\midrule
			1. & 0,0,0,0,0,-2e^{12}-2e^{34}-2\tau e^{57},0 & (\pm [1]_4, \tau, +1, \tau )\\[5pt]
			2. & 0,0,0,0,2 e^{57},-2e^{12}-2e^{34}-2\tau e^{57},0 & (\pm [1]_4, \tau, +1, \tau )\\[5pt]
			3. & 0,0,e^{37},e^{47},2\tau e^{34}+2 e^{57},-2e^{12}-2e^{34}-2\tau e^{57},0 & (\pm [1]_4, \tau, +1, \tau )\\[5pt]
			4. & e^{17},e^{27},e^{37},e^{47},2\tau e^{12}+2\tau e^{34}+2 e^{57},-2e^{12}-2e^{34}-2\tau e^{57},0 & (\pm [1]_4, \tau, +1, \tau )\\[5pt]
			5. & 0,0,0,0,0,-2e^{12}+2e^{34}-2\tau e^{57},0 & (\pm [1]_2,\mp[1]_2, \tau, +1, \tau )\\[5pt]
			6. & 0,0,0,0,2 e^{57},-2e^{12}+2e^{34}-2\tau e^{57},0 & (\pm [1]_2,\mp[1]_2, \tau, +1, \tau )\\[5pt]
			7. & 0,0,e^{37},e^{47},-2\tau e^{34}+2 e^{57},-2e^{12}+2e^{34}-2\tau e^{57},0 & (\pm [1]_2,\mp[1]_2, \tau, +1, \tau )\\[5pt]
			8. & e^{17},e^{27},e^{37},e^{47},2\tau e^{12}-2\tau e^{34}+2 e^{57},-2e^{12}+2e^{34}-2\tau e^{57},0 & (\pm [1]_2,\mp[1]_2, \tau, +1, \tau )\\[5pt]
			\multirow{3}{*}{9.}& \multicolumn{1}{L}{\frac{1}{2}e^{17}+2\lambda e^{27}-\frac12e^{37}-\lambda e^{47},-2\lambda e^{17}+\frac12e^{27}+\lambda e^{37}-\frac12e^{47},} & \multirow{3}{*}{$(\pm [1]_2,\mp[1]_2, \tau, +1, \tau )$}\\
			&\multicolumn{1}{C}{\frac12e^{17}+\lambda e^{27}-\frac12e^{37},-\lambda e^{17}+\frac12e^{27}-\frac12e^{47},}\\
			&\multicolumn{1}{R}{\tau e^{12}-\tau e^{14}+\tau e^{23}+\tau e^{34},-2e^{12}+2e^{34}-2\tau e^{57},0}\\[5pt]
			\multirow{3}{*}{10.} & \multicolumn{1}{L}{\frac12e^{17}+2\lambda e^{27}-\frac32e^{37}-\lambda e^{47},-2\lambda e^{17}+\frac12e^{27}+\lambda e^{37}-\frac32e^{47},} & \multirow{3}{*}{$(\pm [1]_2,\mp[1]_2, \tau, +1, \tau )$}\\
			&\multicolumn{1}{C}{-\frac12e^{17}+\lambda e^{27}-\frac12e^{37},-\lambda e^{17}-\frac12e^{27}-\frac12e^{47},}\\
			&\multicolumn{1}{R}{\tau e^{12}-\tau e^{14}+\tau e^{23}+\tau e^{34}+2e^{57},-2e^{12}+2e^{34}-2\tau e^{57},0}\\[5pt]
			\multirow{3}{*}{11.} & \multicolumn{1}{L}{\frac32e^{17}+2\lambda e^{27}+\frac12e^{37}-\lambda e^{47},-2\lambda e^{17}+\frac32e^{27}+\lambda e^{37}+\frac12e^{47},} & \multirow{3}{*}{$(\pm [1]_2,\mp[1]_2, \tau, +1, \tau )$}\\
			&\multicolumn{1}{C}{\frac32e^{17}+\lambda e^{27}+\frac12e^{37},-\lambda e^{17}+\frac32e^{27}+\frac12e^{47},}\\
			&\multicolumn{1}{R}{3\tau e^{12}-\tau e^{14}+\tau e^{23}-\tau e^{34}+2e^{57},-2e^{12}+2e^{34}-2\tau e^{57},0}\\
			\bottomrule
		\end{tabular}
	}
\end{table}

\begin{theorem}
	\label{thm:central7}
	Let $\tilde\g$ be a Lie algebra of dimension $7$ with a $\lie z$-standard Sasaki structure that reduces to an abelian pseudo-K\"ahler Lie algebra $\check\g$. Then, up to isometry and $\mathcal{D}$-homothety, the metric Lie algebra $(\tilde\g,\tilde{g})$ is one of the Lie algebras appearing in Table~\ref{table:ZstdSasaki7D} and the Sasaki structure is given by
	\[\xi=(e^6)^\flat=e_6,\qquad \eta = e^6, \qquad 2\Phi=d \eta=d e^{6}\]
	with respect to the basis $\{e^1,\dots, e^7\}$ of Table~\ref{table:ZstdSasaki7D}.
\end{theorem}
\begin{proof}
	We first consider the case where $\check\g$ is positive definite, applying Corollary~\ref{cor:diagonalizable} and proceeding as in the proof of Theorem~\ref{thm:central5}.
	
	If $h=0$, we get
	\[(0,0,0,0,0,-2e^{12}-2e^{34}-2\tau e^{57},0);\]
	for $h=2$, we have the three possibilities $\check D=0$, $\check D=e^3\otimes e_3+e^4\otimes e_4$, $\check D=I$, corresponding to
	\begin{gather*}
		(0,0,0,0,2 e^{57},-2e^{12}-2e^{34}-2\tau e^{57},0),\\
		(0,0,e^{37},e^{47},2\tau e^{34}+2 e^{57},-2e^{12}-2e^{34}-2\tau e^{57},0),\\
		(e^{17},e^{27},e^{37},e^{47},2\tau e^{12}+2\tau e^{34}+2 e^{57},-2e^{12}-2e^{34}-2\tau e^{57},0).
	\end{gather*}
	The negative definite case gives rise to the same Lie algebras, with the restriction of the metric to $\check\g$ of opposite sign.
	
	In the neutral case, we can assume
	\begin{gather*}
		\check g=e^1\otimes e^1+e^2\otimes e^2-e^3\otimes e^3-e^4\otimes e^4, \\ \omega=-e^{12}+e^{34}, \\
		J=e^1\otimes e_2 - e^2\otimes e_1+e^3\otimes e_4-e^4\otimes e_3.
	\end{gather*}
	If $\check D^s$ is diagonalizable,
	Corollary~\ref{cor:diagonalizable} applies and computations as above yield
	\begin{gather*}
		(0,0,0,0,0,-2e^{12}+2e^{34}-2\tau e^{57},0),\\
		(0,0,0,0,2 e^{57},-2e^{12}+2e^{34}-2\tau e^{57},0),\\
		(0,0,e^{37},e^{47},-2\tau e^{34}+2 e^{57},-2e^{12}+2e^{34}-2\tau e^{57},0),\\
		(e^{17},e^{27},e^{37},e^{47},2\tau e^{12}-2\tau e^{34}+2 e^{57},-2e^{12}+2e^{34}-2\tau e^{57},0).
	\end{gather*}
	If $\check D^s$ is not diagonalizable, we can exploit the $U(1,1)$ symmetry preserving the pseudo-K\"ahler structure of $\check\g$. Indeed, a symmetric derivation commuting with $J$ is effectively an element of $i\lie{u}(1,1)$, with $U(1,1)$ acting on it by the adjoint action. Write $\check D^s=tI+\check D^s_0$, where $\check D^s_0$ is traceless. Then $\check D^s_0$ can therefore be viewed as an element of $i\su(1,1)$. Now $\SU(1,1)$ is isomorphic to $\SL(2,\R)$ via the Cayley isomorphism
	\begin{equation}
		\label{eqn:Cayley}
		\SL(2,\R)\ni g\mapsto CgC^{-1}\in\SU(1,1),
	\end{equation}
	where $C=\begin{pmatrix} 1& -i \\ 1 & i \end{pmatrix}$. The action of $\SL(2,\R)$ on its Lie algebra is conjugation, so any nondiagonalizable element of $\Sl(2,\R)$ is in the $\SL(2,\R)$-orbit of
	$\begin{pmatrix} 0 & 1 \\ 0 & 0 \end{pmatrix}$. Reading this in $\su(1,1)$ via \eqref{eqn:Cayley} and multiplying by $-i$, we see that $\check D^s_0$ corresponds to the complex matrix $\begin{pmatrix} 1/2& -1/2 \\ 1/2 &-1/2 \end{pmatrix}$; writing it as a real matrix, we obtain
	\[\check D^s=\begin{pmatrix}(t+\frac12)I& -\frac12 I \\ \frac12 I & (t-\frac12) I\end{pmatrix}.\]
	A derivation $\check{D}$ that satisfies $[D,J]=0$ and is not diagonalizable takes the form
	\[
	\check D=\begin{pmatrix}
		x&\lambda_2&\lambda_5-1&-\lambda_6\\
		-\lambda_2&x&\lambda_6&\lambda_5-1\\
		\lambda_5&\lambda_{6}&x-1&\lambda_8\\
		-\lambda_6&\lambda_5&-\lambda_8&x-1
	\end{pmatrix}.
	\]
	Now, thanks to Proposition \ref{prop:pseudoAzencottWilson}, we can consider any
	\[
	\check D'=\begin{pmatrix}
		y&\mu_2&\mu_5-1&-\mu_6\\
		-\mu_2&y&\mu_6&\mu_5-1\\
		\mu_5&\mu_{6}&y-1&\mu_8\\
		-\mu_6&\mu_5&-\mu_8&y-1
	\end{pmatrix}.
	\]
	such that $[\check{D}',\check{D}]=0$ and $\check{D}'^s=\check{D}^s$. This yields $y=x$, $\mu_5=\lambda_5$, $\mu_6=\lambda_6$ and $\mu_2-\mu_8=\lambda_2-\lambda_8$, hence we can consider $\check D$ to be
	\[
	\check D=\begin{pmatrix}
		x&\lambda_2&\lambda_5-1&-\lambda_6\\
		-\lambda_2&x&\lambda_6&\lambda_5-1\\
		\lambda_5&\lambda_{6}&x-1&0\\
		-\lambda_6&\lambda_5&0&x-1
	\end{pmatrix}.
	\]
	Again we distinguish two cases depending on $h$.
	
	If $h=0$ then equation $[\check D^s,\check D^a]=h\check D^s-2(\check D^s)^2$ yields
	\[
	\check{D}=\begin{pmatrix}
		\frac{1}{2}&2\lambda&-\frac12&-\lambda\\
		-2\lambda&\frac{1}{2}&\lambda&-\frac12\\
		\frac12&\lambda&-\frac{1}{2}&0\\
		-\lambda&\frac12&0&-\frac12
	\end{pmatrix}.
	\]
	Hence we set $d\xi^\flat=-2e^{12}+2e^{34}$, $db^\flat=\tau e^{12}-\tau e^{14}+\tau e^{23}+\tau e^{34}$, and the first Lie algebra extension is
	\[
	\mathfrak{g}=(0, 0, 0, 0, \tau e^{12}-\tau e^{14}+\tau e^{23}+\tau e^{34}, -2e^{12}+2e^{34}),
	\]
	with metric
	\begin{equation}
		\label{eqn:metrich0}
		g = e^1\otimes e^1+e^2\otimes e^2-e^3\otimes e^3-e^4\otimes e^4+\tau b^\flat\otimes b^\flat+\xi^\flat\otimes\xi^\flat.
	\end{equation}
	The Sasaki extension $\tilde{\mathfrak{g}}=\mathfrak{g}\rtimes\Span{e_0}$ is determined by
	\begin{gather*}
		d\xi^\flat=-2e^{12}+2e^{34},\qquad db^\flat=\tau e^{12}-\tau e^{14}+\tau e^{23}+\tau e^{34}\\
		[e_0,x]=\check{D}x,\qquad[e_0,\xi]=0,\qquad[e_0,b]=-2\tau\xi;
	\end{gather*}
	hence the Lie algebra is
	\begin{gather*}\tilde{\mathfrak{g}}=(\frac{1}{2}e^{17}+2\lambda e^{27}-\frac12e^{37}-\lambda e^{47},-2\lambda e^{17}+\frac12e^{27}+\lambda e^{37}-\frac12e^{47},\\
		\frac12e^{17}+\lambda e^{27}-\frac12e^{37},-\lambda e^{17}+\frac12e^{27}-\frac12e^{47},\\
		\tau e^{12}-\tau e^{14}+\tau e^{23}+\tau e^{34},-2e^{12}+2e^{34}-2\tau e^{57},0).
	\end{gather*}
	If $h=2$ then equation $[\check D^s,\check D^a]=h\check D^s-2(\check D^s)^2$ yields two distinct solutions for $\check{D}$:
	\[
	\check{D}_1=
	\begin{pmatrix}
		\frac12 & 2\lambda &-\frac32&-\lambda\\
		-2\lambda & \frac12& \lambda & -\frac32\\
		-\frac12&\lambda& -\frac12&0\\
		-\lambda&-\frac12&0&-\frac12
	\end{pmatrix}\quad\text{or}\quad\check{D}_2=
	\begin{pmatrix}
		\frac32 & 2\lambda &\frac12&-\lambda\\
		-2\lambda & \frac32& \lambda & \frac12\\
		\frac32&\lambda& \frac12&0\\
		-\lambda&\frac32&0&\frac12
	\end{pmatrix}.
	\]
	For $\check{D}_1$ we get $db^\flat=\tau e^{12}-\tau e^{14}+\tau e^{23}+\tau e^{34}$, hence
	\[
	\mathfrak{g}=(0,0,0,0,\tau e^{12}-\tau e^{14}+\tau e^{23}+\tau e^{34},-2e^{12}+2e^{34});
	\]
	for $\check{D}_2$ we get $db^\flat=3\tau e^{12}-\tau e^{14}+\tau e^{23}-\tau e^{34}$ and
	\[
	\mathfrak{g}=(0,0,0,0,3\tau e^{12}-\tau e^{14}+\tau e^{23}-\tau e^{34},-2e^{12}+2e^{34}).
	\]
	In both cases, the metric is given by \eqref{eqn:metrich0}. The resulting Lie algebras $\tilde{\mathfrak{g}}$ correspond to n. $10$ and n. $11$ in Table~\ref{table:ZstdSasaki7D}.
\end{proof}
\begin{remark}
	Each Lie algebra in rows 1--4 of Table~\ref{table:ZstdSasaki7D} is isomorphic to a Lie algebra in rows 5--8 under the transformation $e_3\mapsto -e_3$. This isomorphism is not an isometry; indeed, there is no isometry preserving the standard decomposition $\Span{e_1,\dotsc, e_6}\rtimes\Span{e_7}$, because the nilpotent factors $\Span{e_1,\dotsc, e_6}$ have different signatures. Notice also that, in each row, Lie algebras corresponding to opposite values of $\tau$ are isomorphic under the transformation $e_5\mapsto -e_5$.
\end{remark}

\medskip
\small\noindent D. Conti: Dipartimento di Matematica, Università di Pisa, largo B. Pontecorvo 6, 56127 Pisa, Italy.\\
\texttt{diego.conti@unipi.it}\\
\small\noindent R.~Segnan Dalmasso: Dipartimento di Matematica e Applicazioni, Universit\`a di Milano Bicocca, via Cozzi 55, 20125 Milano, Italy.\\
\texttt{r.segnandalmasso@campus.unimib.it}\\
\small\noindent F. A. Rossi: Dipartimento di Matematica e Informatica, Universit\`a degli studi di Perugia, via Vanvitelli 1, 06123 Perugia, Italy.\\
\texttt{federicoalberto.rossi@unipg.it}

%\bibliographystyle{plainurl}
%
%\bibliography{PseudoSasakiAndKahler}

\begin{thebibliography}{99}
	
	\bibitem{AndFinVez:Sasaki5}
	A.~Andrada, A.~Fino, and L.~Vezzoni.
	\newblock A class of {S}asakian 5-manifolds.
	\newblock {\em Transform. Groups}, 14(3):493--512, 2009.
	\newblock \href {https://doi.org/10.1007/s00031-009-9058-9}
	{\path{doi:10.1007/s00031-009-9058-9}}.
	
	\bibitem{Baum:TwistorAndKilling}
	H.~Baum.
	\newblock Twistor and {K}illing spinors in {L}orentzian geometry.
	\newblock In {\em Global analysis and harmonic analysis ({M}arseille-{L}uminy,
		1999)}, volume~4 of {\em S\'{e}min. Congr.}, pages 35--52. Soc. Math. France,
	Paris, 2000.
	
	\bibitem{BoucettaTibssirte}
	M.~Boucetta and O.~Tibssirte.
	\newblock On {E}instein {L}orentzian nilpotent {L}ie groups.
	\newblock {\em J. Pure Appl. Algebra}, 224(12):106443, 22, 2020.
	\newblock \href {https://doi.org/10.1016/j.jpaa.2020.106443}
	{\path{doi:10.1016/j.jpaa.2020.106443}}.
	
	\bibitem{BoGa:SasakianBook}
	C.~P. Boyer and K.~Galicki.
	\newblock {\em Sasakian geometry}.
	\newblock Oxford Mathematical Monographs. Oxford University Press, Oxford,
	2008.
	
	\bibitem{CalvCaLo:PseudoRiemHomBook}
	G.~Calvaruso and M.~Castrill\'{o}n~L\'{o}pez.
	\newblock {\em Pseudo-{R}iemannian homogeneous structures}, volume~59 of {\em
		Developments in Mathematics}.
	\newblock Springer, Cham, 2019.
	\newblock \href {https://doi.org/10.1007/978-3-030-18152-9}
	{\path{doi:10.1007/978-3-030-18152-9}}.
	
	\bibitem{ContiRossi:IndefiniteNilsolitons}
	D.~Conti and F.~A. Rossi.
	\newblock Indefinite nilsolitons and {E}instein solvmanifolds.
	\newblock {\em J. Geom. Anal.}, 32(3):88, 2022.
	\newblock \href {https://doi.org/10.1007/s12220-021-00850-7}
	{\path{doi:10.1007/s12220-021-00850-7}}.
	
	\bibitem{ContiRossi:NiceNilsolitons}
	D.~Conti and F.~A. Rossi.
	\newblock Nice pseudo-{R}iemannian nilsolitons.
	\newblock {\em J. Geom. Phys.}, 173:104433, 2022.
	\newblock \href {https://doi.org/10.1016/j.geomphys.2021.104433}
	{\path{doi:10.1016/j.geomphys.2021.104433}}.
	
	\bibitem{ConDal:KillingSpinHyper}
	D.~Conti and R.~Segnan~Dalmasso.
	\newblock {K}illing spinors and hypersurfaces.
	\newblock arXiv:2111.13202 [math.DG].
	
	\bibitem{Duff1986Kaluza-KleinSupergravity}
	M.~Duff, B.E.W. Nilsson, and C.N. Pope.
	\newblock {Kaluza-Klein supergravity}.
	\newblock {\em Physics Reports}, 130(1-2):1--142, 1 1986.
	\newblock \href {https://doi.org/10.1016/0370-1573(86)90163-8}
	{\path{doi:10.1016/0370-1573(86)90163-8}}.
	
	\bibitem{EberleinHeber}
	P.~Eberlein and J.~Heber.
	\newblock Quarter pinched homogeneous spaces of negative curvature.
	\newblock {\em Internat. J. Math.}, 7(4):441--500, 1996.
	\newblock \href {https://doi.org/10.1142/S0129167X96000268}
	{\path{doi:10.1142/S0129167X96000268}}.
	
	\bibitem{Hatakeyama:SomeNotes}
	Y.~Hatakeyama.
	\newblock Some notes on differentiable manifolds with almost contact
	structures.
	\newblock {\em T\^{o}hoku Math. J. (2)}, 15:176--181, 1963.
	\newblock \href {https://doi.org/10.2748/tmj/1178243844}
	{\path{doi:10.2748/tmj/1178243844}}.
	
	\bibitem{Heber:noncompact}
	J.~Heber.
	\newblock Noncompact homogeneous {E}instein spaces.
	\newblock {\em Invent. Math.}, 133(2):279--352, 1998.
	\newblock \href {https://doi.org/10.1007/s002220050247}
	{\path{doi:10.1007/s002220050247}}.
	
	\bibitem{Lauret:Einstein_solvmanifolds}
	J.~Lauret.
	\newblock Einstein solvmanifolds are standard.
	\newblock {\em Ann. of Math. (2)}, 172(3):1859--1877, 2010.
	\newblock \href {https://doi.org/10.4007/annals.2010.172.1859}
	{\path{doi:10.4007/annals.2010.172.1859}}.
	
	\bibitem{Ogiue:OnFiberings}
	K.~Ogiue.
	\newblock {On fiberings of almost contact manifolds}.
	\newblock {\em Kodai Mathematical Seminar Reports}, 17(1):53 -- 62, 1965.
	\newblock \href {https://doi.org/10.2996/kmj/1138845019}
	{\path{doi:10.2996/kmj/1138845019}}.
	
	\bibitem{Salamon:ComplexStructures}
	S.~M. Salamon.
	\newblock Complex structures on nilpotent {L}ie algebras.
	\newblock {\em J. Pure Appl. Algebra}, 157(2-3):311--333, 2001.
	\newblock \href {https://doi.org/10.1016/S0022-4049(00)00033-5}
	{\path{doi:10.1016/S0022-4049(00)00033-5}}.
	
	\bibitem{SasakiHatakeyama}
	S.~Sasaki and Y.~Hatakeyama.
	\newblock On differentiable manifolds with contact metric structures.
	\newblock {\em J. Math. Soc. Japan}, 14:249--271, 1962.
	\newblock \href {https://doi.org/10.2969/jmsj/01430249}
	{\path{doi:10.2969/jmsj/01430249}}.
	
	\bibitem{Tak:PseudoSasaki}
	T.~Takahashi.
	\newblock Sasakian manifold with pseudo-{R}iemannian metric.
	\newblock {\em Tohoku Math. J. (2)}, 21:271--290, 1969.
	\newblock \href {https://doi.org/10.2748/tmj/1178242996}
	{\path{doi:10.2748/tmj/1178242996}}.
	
	\bibitem{Walker1970OnSpacetimes}
	M.~Walker and R.~Penrose.
	\newblock {On quadratic first integrals of the geodesic equations for type
		{\{}22{\}} spacetimes}.
	\newblock {\em Communications in Mathematical Physics}, 18(4):265--274, 12
	1970.
	\newblock \href {https://doi.org/10.1007/BF01649445}
	{\path{doi:10.1007/BF01649445}}.
	
\end{thebibliography}

\end{document}